\documentclass[10pt,reqno]{amsart}

\usepackage{amsthm, mathrsfs, amsmath, amstext, amsxtra, amsfonts, dsfont, amssymb}
\usepackage{lmodern}
\usepackage{relsize}
\usepackage{soul}
\usepackage[dvipsnames]{xcolor}
\usepackage[colorlinks, linkcolor=red, citecolor=blue, urlcolor=OliveGreen, 
hypertexnames=false]{hyperref}

\usepackage[belowskip=2pt,aboveskip=0pt]{caption}
\usepackage{booktabs}
\usepackage{geometry}
\newcommand{\Z}{\mathbb Z} 
\newcommand{\R}{\mathbb R} 
\newcommand{\C}{\mathbb C} 

\newcommand{\varep}{\varepsilon}

\newcommand{\scal}[1]{\left\langle #1 \right\rangle} 

\newtheorem{theorem}{Theorem}[section]
\newtheorem{lemma}[theorem]{Lemma}

\newtheorem{claim}[theorem]{Claim} 
\theoremstyle{definition}
\newtheorem{definition}[theorem]{Definition}
\newtheorem{remark}[theorem]{Remark}

\title[Existence and blow-up behavior of minimizers]{Existence, non-existence and blow-up behavior of minimizers for the mass-critical fractional nonlinear Schr\"odinger equations with periodic potentials} 

\author[V. D. Dinh]{Van Duong Dinh}
\address[V. D. Dinh]{Laboratoire Paul Painlev\'e UMR 8524, Universit\'e de Lille CNRS, 59655 Villeneuve d'Ascq Cedex, France
	and 
Department of Mathematics, HCMC University of Pedagogy, 280 An Duong Vuong, Ho Chi Minh, Vietnam}
\email{contact@duongdinh.com}

\keywords{Fractional nonlinear Schr\"odinger equation; Standing waves; Concentration-compactness principle; Blow-up profile; Periodic potential}

\subjclass[2010]{35A15, 35B44, 35J35, 35Q55}

\begin{document}

\begin{abstract}
We consider the minimizing problem for the energy functional with prescribed mass constraint related to the fractional nonlinear Schr\"odinger equation with periodic potentials. Using the concentration-compactness principle, we show a complete classification for the existence and non-existence of minimizers for the problem. In the mass-critical case, under a suitable assumption of the potential, we give a detailed description of blow-up behavior of minimizers once the mass tends to a critical value.
\end{abstract}

\maketitle


\section{Introduction} \label{section introduction}
\setcounter{equation}{0}

In this paper, we consider the following time-independent fractional nonlinear Schr\"odinger equation
\begin{align} \label{FNLS}
(-\Delta)^s u + Vu = |u|^\alpha u, \quad x \in \R^d,
\end{align}
where $d\geq 1$, $0<s<1$, $\alpha>0$ and $V: \R^d \rightarrow \R$ is an external potential. The operator $(-\Delta)^s$ is the fractional Laplacian operator which is defined by $(-\Delta)^s u = \mathcal{F}^{-1}[|\xi|^{2s} \mathcal{F}u]$, where $\mathcal{F}$ and $\mathcal{F}^{-1}$ are the Fourier transform and inverse Fourier transform respectively. 

The problem \eqref{FNLS} comes from the following time dependent Schr\"odinger-type equation
\begin{align} \label{time-FNLS}
i\partial_t \psi = (-\Delta)^s \psi + (V - \omega) \psi  - |\psi|^\alpha \psi, \quad (t,x) \in \R \times \R^d
\end{align}
by looking for standing wave solutions $\psi(t,x) = e^{i \omega t} u(x)$ with $\omega \in \R$ a frequency. The equation \eqref{time-FNLS} is the fractional nonlinear Schr\"odinger equation which was introduced by Laskin \cite{Laskin1, Laskin2} as a result of extending the Feynman path integral, from the Brownian-like to L\'evy-like quantum mechanical paths. The equation \eqref{time-FNLS} also appears in the continuum limit of discrete models with long-range lattice interactions (see e.g. \cite{KLS}) and in the description of Boson stars (see e.g \cite{FJL}) as well as in water wave dynamics (see e.g. \cite{IP}).

The equation \eqref{FNLS} involves the fractional Laplacian $(-\Delta)^s, 0<s<1$ which is a non-local operator. A general approach to deal with this problem due to Caffarelli-Silvestre \cite{CS} is to transform \eqref{FNLS} into a local one  via the Dirichlet-Neumann map. That is, one consider the extension $U: \R^d \times [0,\infty) \rightarrow \R$ which satisfies
\[
\left\{
\begin{array}{rcl}
-\text{div}(y^{1-2s} \nabla U) &=& 0 \quad \text{in } \R^d \times [0,\infty), \\
U(y=0) &=& u \quad \text{on } \R^d.
\end{array}
\right.
\]
It was shown in \cite{CS} that
\[
(-\Delta)^s u(x) = A(d,s) \lim_{y\rightarrow 0^+} - y^{1-2s} U_y(x,y)
\]
and 
\[
\|(-\Delta)^s u\|^2_{L^2} = \int_{\R^d} |\xi|^{2s} |\mathcal{F}u (\xi)|^2 d\xi = A(d,s) \iint_{\R^d \times [0,\infty)} |\nabla U(x,y)|^2 y^{1-2s} dx dy,
\]
where $A(d,s)$ is an appropriate constant depending on $d$ and $s$. This method has been applied successfully to study equations involving the fractional Laplacian, and a series of significant results have been obtained (see e.g. \cite{CSi1}, \cite{CSi2}, \cite{CSS}, \cite{CRS}, \cite{SV} and references therein).

In the sequel, we are interested in the existence, non-existence and blow-up behavior of minimizers for the energy functional related to \eqref{FNLS} under the prescribed mass constraint. More precisely, we consider the minimizing problem: for $a>0$,
\[
I(a):= \inf \left\{ E(u) \ : \ u \in \mathcal{H}(\R^d), \|u\|^2_{L^2} =a \right\},
\]
where 
\[
E(u):= \frac{1}{2} \int_{\R^d} |(-\Delta)^{\frac{s}{2}} u(x)|^2 dx+ \frac{1}{2} \int_{\R^d} V(x)|u(x)|^2 dx - \frac{1}{\alpha+2} \int_{\R^d} |u(x)|^{\alpha+2} dx
\]
and the exponent $\alpha$ satisfies $0<\alpha <s^*$ with
\begin{align} \label{defi-alpha-sup}
s^*:= \left\{
\begin{array}{cl}
\frac{4s}{d-2s} &\text{if } d>2s, \\
\infty &\text{if } d \leq 2s.
\end{array}
\right.
\end{align}
Here $\mathcal{H}(\R^d)$ is a subspace of $H^s(\R^d)$ which is defined by
\[
\mathcal{H}(\R^d):= \left\{ u \in H^s(\R^d) \ : \ \int_{\R^d} V|u|^2 dx <\infty \right\}.
\]
Note that the space $\mathcal{H}(\R^d)$ is a Hilbert space with scalar product and norm
\[
\scal{u,v}_{\mathcal{H}} := \int_{\R^d} u \overline{v} dx + \int_{\R^d} (-\Delta)^{\frac{s}{2}} u \overline{(-\Delta)^{\frac{s}{2}} v} dx + \int_{\R^d} V u \overline{v} dx, \quad \|u\|_{\mathcal{H}} = \sqrt{\scal{u,u}_{\mathcal{H}}}.
\]

Let us recall some known results on the existence, non-existence and blow-up behavior of minimizers for the energy functional with prescribed mass constraint. In the case $s=1$, i.e. the classical nonlinear Schr\"odinger equation, the problem has been studied by many mathematicians. 

In the case of constant potentials, using the symmetric rearrangement argument or the concentration-compactness principle of Lions \cite{Lions1, Lions2}, it holds that (see e.g. \cite[Proposition 8.3.6]{Cazenave}) for any $a>0$, there exists at least a minimizer for $I(a)$ in the mass-subcritical case $0<\alpha<\frac{4}{d}$. In the mass-critical case $\alpha=\frac{4}{d}$, one can prove that for any $a>0$ and $a \ne \|R\|^2_{L^2}$, there is no minimizer for $I(a)$; and for $a=\|R\|^2_{L^2}$, there is a unique (up to symmetries) minimizer for $I(a)$, where $R$ is the unique (up to symmetries) positive radial solution to 
\begin{align} \label{ell-NLS-mass}
-\Delta R + R - |R|^{\frac{4}{d}} R=0.
\end{align}

In the case of harmonic potential, i.e. $V = |x|^2$, using the compact embedding $\mathcal{H} \hookrightarrow L^q$ for any $2\leq q<\frac{2d}{d-2}$ if $d\geq 3$ (or $2\leq q<\infty$ if $d=1,2$), Zhang \cite{Zhang} proved that for any $a>0$, there exists at least a minimizer for $I(a)$ in the mass-subcritical case. In the mass-critical case, he proved that for $0<a<\|R\|^2_{L^2}$, there exists at least a minimizer for $I(a)$. 

In \cite{GS}, Guo-Seiringer studied the minimizing problem
\begin{align} \label{min-pro-Jb}
J(b) = \inf \left\{ E_b(u) \ : \ u \in \mathcal{H}(\R^2), \|u\|^2_{L^2}=1 \right\},
\end{align}
where $b>0$ is a given parameter, 
\[
E_b(u) = \int_{\R^2} |\nabla u(x)|^2 dx + \int_{\R^2} V(x)|u(x)|^2 dx - \frac{b}{2} \int_{\R^2} |u(x)|^4 dx
\]
and $V$ is a trapping potential, i.e.
\begin{align} \label{trapping}
0 \leq V \in L^\infty_{\text{loc}}(\R^2), \quad \lim_{|x| \rightarrow \infty} V(x) = \infty.
\end{align}
Note that by a simple scaling argument, it is easy to see that $I(a)=\frac{a}{2} J(a)$. Using the same argument as in \cite{Zhang}, they proved that $J(b)$ has at least a minimizer if $0<b<b^*:=\|R\|^2_{L^2}$, whereas there is no minimizer for $J(b)$ if $b \geq b^*$. They also proved that if $u_b$ is a minimizer for $J(b)$ with $0<b<b^*$, then $u_b$ blows up as $b \nearrow b^*$ in the sense that
\begin{align} \label{ub}
\lim_{b\nearrow b^*} \|\nabla u_b\|_{L^2}=\infty.
\end{align}
Moreover, they gave a detailed description of the blow-up behavior of minimizers for $J(b)$ by assuming that the trap potential $V$ has a finite number of isolated minima, and that in their vicinity, $V$ behaves like a power of the distance from these points. For instance, if $V(x)= \kappa |x-x_0|^p$ for some $\kappa, p>0$, then as $b\nearrow b^*$,
\[
(b^*-b)^{\frac{1}{2+p}} u_b \left( (b^*-b)^{\frac{1}{2+p}} \cdot + x_0 \right) \rightarrow \lambda_0 R_0(\lambda_0 \cdot) \text{ strongly in } L^q(\R^2)
\]
for any $2\leq q<\infty$, where
\begin{align} \label{scaling-R}
R_0 = \frac{R}{\|R\|_{L^2}}, \quad \lambda_0 = \left( \frac{\kappa p}{2} \int_{\R^2} |x|^p [R(x)]^2 dx \right)^{\frac{1}{2+p}}.
\end{align}

In the case of bounded potentials satisfying 
\begin{align} \label{bounded}
V \in C^1(\R^d), \quad 0 = \inf_{\R^d} V <\sup_{\R^d} V = \lim_{|x| \rightarrow \infty} V(x) <\infty,
\end{align}
by using the concentration-compactness principle of Lions, one can prove (see e.g. \cite{Maeda}) that for any $a>0$, there exists at least a minimizer for $I(a)$ in the mass-subcritical case. In the mass-critical case, it can be proved that for any $0<a<\|R\|^2_{L^2}$, there exists at least a minimizer for $I(a)$. Moreover, by the same argument in \cite{GS}, there is no mimizer for $I(a)$ with $a\geq \|R\|^2_{L^2}$. In \cite{Maeda}, Maeda studied the uniqueness, concentration and symmetry of minimizers for $I(a)$ as $a \rightarrow \infty$ in the mass-subcritical case. 

In \cite{WZ}, Wang-Zhao studied the existence and non-existence of minimizers for $J(b)$ (see \eqref{min-pro-Jb}) with continuous periodic potentials satisfying 
\[
0 = \min V < \inf \sigma(-\Delta+V),
\]
where
\[
\inf \sigma (-\Delta +V) := \inf \left\{ \int_{\R^2} |\nabla u(x)|^2 dx + \int_{\R^2} V(x)|u(x)|^2 dx  \ : \ u \in H^1(\R^2), \|u\|^2_{L^2} =1 \right\}.
\]
Using the concentration-compactness principle of Lions, they proved the existence of minimizers for $J(b)$ for $b_*<b<b^*$ with some $0<b_*<b^*$ (see above \eqref{ub} for the definition of $b^*$). On the other hand, there is no minimizer for $J(b)$ when $b\geq b^*$. Moreover, under the assumption 
\[
V^{-1}(0) = x_0 + \Z^2 \text{ for some } x_0 \in [0,1]^2, \quad V(x) = O(|x-x_0|^p) \text{ as } x \rightarrow x_0 \text{ for some } p>0,
\] 
they proved that if $u_b$ is a minimizer for $J(b)$ with $b_*<b<b^*$, then as $b \nearrow b^*$,
\[
(b^*-b)^{\frac{1}{2+p}} u_b \left( (b^*-b)^{\frac{1}{2+p}} \cdot + x_0 \right) \rightarrow \lambda_0 R_0(\lambda_0 \cdot) \text{ strongly in } L^q(\R^2)
\]
for any $2\leq q<\infty$, where $R_0$ and $\lambda_0$ are as in \eqref{scaling-R}. 

In \cite{Phan}, Phan studied the existence and non-existence of minimizers for $J(b)$ with attractive potentials satisfying 
\[
0\geq V \in L^q(\R^2) + L^r(\R^2), \quad \inf \sigma(-\Delta +V)<0
\]
for some $1<q<r<\infty$. He proved that for any $0<b<b^*$, there exists at least a minimizer for $J(b)$, and there is no minimizer for $J(b)$ if $b\geq b^*$. Moreover, under a suitable assumption on the external potential, he gave a detailed description of the blow-up behavior of minimizers for $J(b)$ as $b \nearrow b^*$. The result implies in particular that if $V(x) = - \kappa |x-x_0|^{-p}$ for some $\kappa>0$ and $0<p<2$, then as $b \nearrow b^*$,
\[
(b^*-b)^{\frac{1}{2-p}} u_b \left( (b^*-b)^{\frac{1}{2-p}} \cdot + x_0\right) \rightarrow \lambda_0 R_0(\lambda_0 \cdot) \text{ strongly in } H^1(\R^2),
\]
where 
\[
\lambda_0 = \left( \frac{\kappa p}{2} \int_{\R^2} |x|^{-p} [R(x)]^2 dx \right)^{\frac{1}{2-p}}.
\]

In the case of an inverse-square potential $V(x)= c |x|^{-2}$ with $c>-\left(\frac{d-2}{2}\right)^2$, using the profile decomposition, Bensouilah-Dinh-Zhu \cite{BDZ} proved the existence of minimizers for $I(a)$ for any $a>0$ in the mass-subcritical case.

The existence, non-existence and blow-up behavior of minimizers for $J(b)$ has been extended to ring-shaped potentials in \cite{GZZ}, multi-well potentials \cite{GWZZ}, ellipse-shaped potentials \cite{GZ} and rotating trap potentials \cite{GLY}.

In the case $0<s<1$, the existence, non-existence and blow-up behavior of minimizers for $I(a)$ has been considered in several works. In \cite{HL}, He-Long proved the existence and non-existence of minimizers for $I(a)$ with trapping potentials \eqref{trapping} and bounded potentials satisfying \eqref{bounded}. They also studied the blow-up behavior of minimizers for $I(a)$ as $a$ tends to a critical value in the mass-critical case $\alpha=\frac{4s}{d}$. 

Recently, Du-Tian-Wang-Zhang \cite{DTWZ} gave a complete classification of the existence and non-existence of minimizers for $I(a)$ with trapping potentials \eqref{trapping} for $0<\alpha<s^*$ (see \eqref{defi-alpha-sup} for the definition of $s^*$). Moreover, under a suitable assumption of the external potential, they showed a detailed analysis of the blow-up behavior of minimizers for $I(a)$ in the mass-critical case $\alpha=\frac{4s}{d}$. 

Motivated by the aforementioned papers, we study the existence and non-existence of minimizers for the energy functional related to \eqref{FNLS} with the prescribed mass constraint. In this paper, we focus mainly on the fractional Schr\"odinger equation with periodic potentials. Before stating our main results, we introduce the following notion of ground states related to the fractional Schr\"odinger equation.

\begin{definition} [Ground states] Let $d\geq 1$, $0<s<1$ and $0<\alpha <s^*$. A non-zero, non-negative $H^s$ function $Q_\alpha$ is called a {\bf ground state} related to 
	\begin{align} \label{ell-equ-fra}
	(-\Delta)^s u + u - |u|^\alpha u =0
	\end{align}
	if it solves \eqref{ell-equ-fra} and is a minimizer for the Weinstein's functional
	\begin{align} \label{weinstein-functional}
	J(u):= \left[\|(-\Delta)^{\frac{s}{2}} u\|^{\frac{d\alpha}{2s}}_{L^2} \|u\|^{\alpha+2-\frac{d\alpha}{2s}}_{L^2} \right] \div \|u\|^{\alpha+2}_{L^{\alpha+2}},
	\end{align}
	that is,
	\[
	J(Q_\alpha) = \inf \left\{J(u) \ : \ u \in H^s(\R^d) \backslash \{0\} \right\}.
	\]
\end{definition}

The existence, uniqueness, symmetry, regularity and decay of the ground state related to \eqref{ell-equ-fra} has been established in celabrated papers \cite{FL, FLS} (see Theorem $\ref{theo-fractional-GN}$ more details).

From now on, we denote 
\begin{align} \label{defi-a*}
a^*:= \|Q\|^2_{L^2},
\end{align}
where $Q:= Q_{\frac{4s}{d}}$ is the unique (up to translations) positive radial ground state related to 
\begin{align} \label{ell-equ-fra-cri}
(-\Delta)^s u + u - |u|^{\frac{4s}{d}} u =0.
\end{align}

Our first result is the following existence and non-existence of minimizers for $I(a)$ in the case of no external potential $V \equiv 0$.
\begin{theorem} [No potential] \label{theo-no-potential}
	Let $d\geq 1$ and $0<s<1$. Let $V \equiv 0$. Then it holds that:
	\begin{itemize}
		\item \cite{Feng} If $0<\alpha<\frac{4s}{d}$, then for any $a>0$, there exits at least a minimizer for $I(a)$ and $-\infty<I(a)<0$.
		\item If $\alpha=\frac{4s}{d}$, then for any $a>0$ and $a \ne a^*$,	there is no minimizer for $I(a)$; and for $a=a^*$, there is a unique (up to symmetries) minimizer for $I(a)$. Moreover, $I(a) \geq 0$ if $0<a<a^*$, $I(a^*)=0$ and $I(a) = -\infty$ if $a>a^*$.
		\item If $\frac{4s}{d}<\alpha<s^*$, then for any $a>0$, there is no minimizer for $I(a)$ and $I(a) = -\infty$. 
	\end{itemize}
\end{theorem}
Since adding a constant in the potential does not change the minimizing problem, the above result still holds in the case of constant potentials. In the mass-subcritical case $0<\alpha<\frac{4s}{d}$, the existence of minimizers for $I(a)$ has been studied in \cite{Feng} by using a fractional version of the concentration-compactness principle of P. L. Lions \cite{Lions1, Lions2}. We will give an alternative proof using the symmetric rearrangement argument. We refer the reader to Section $\ref{section-existence}$ for more details.

Our next result is the following existence and non-existence of minimizers for $I(a)$ in the case of periodic potentials.

\begin{theorem}[Periodic potential] \label{theo-existence}
	Let $d\geq 1$ and $0<s<1$. Let $V \in C(\R^d)$ be such that
	\begin{itemize}
		\item[(V1)] $V(x+z) = V(x)$ for all $x\in \R^d$ and $z \in \Z^d$.
	\end{itemize}
	Then it holds that:
	\begin{itemize}
		\item If $0<\alpha<\frac{4s}{d}$, then there exists $a_*>0$ such that for any $a>a_*$, there exists at least a minimizer for $I(a)$ and $-\infty<I(a)<\frac{a}{2} \inf \sigma((-\Delta)^s+V)$.
		\item If $\alpha=\frac{4s}{d}$ and assume in addition that
		\begin{itemize}
			\item[(V2)] $\min V < \inf \sigma((-\Delta)^s+V)$, where $\inf \sigma((-\Delta)^s+V)$ is the infimum of the spectrum of $(-\Delta)^s+V$, i.e.
			\[
			\inf \sigma((-\Delta)^s+V) = \inf \left\{ \int_{\R^d} |(-\Delta)^{\frac{s}{2}} u(x)|^2 dx + \int_{\R^d} V(x)|u(x)|^2 dx \ : \ \|u\|_{L^2} =1 \right\},
			\]
		\end{itemize}
		 then there exists $0<a_*<a^*$ such that for any $a_*<a<a^*$, there exists at least a minimizer for $I(a)$; and for $a\geq a^*$, there is no minimizer for $I(a)$. Moreover, $\frac{a}{2} \min V < I(a)<\frac{a}{2} \inf \sigma((-\Delta)^s +V)$ if $a_*<a<a^*$, $I(a^*)=\frac{a^*}{2} \min V$ and $I(a)=-\infty$ if $a>a^*$.
		\item If $\frac{4s}{d}<\alpha<s^*$, then for any $a>0$, there is no minimizer for $I(a)$ and $I(a)=-\infty$.
	\end{itemize}
\end{theorem}
This result gives a complete classification of the existence and non-existence of minimizers for $I(a)$ with periodic potentials. The proof is based on the concentration-compactness principle in the same spirit of Lions \cite{Lions1, Lions2}. However, since we are dealing with the non-local opeartor $(-\Delta)^s$, we need a careful analysis in order to get this concentration-compactness principle. We refer the reader to Section $\ref{section-preliminaries}$ for more details.

Our next result is the blow-up behavior of minimizers for $I(a)$ as $a \nearrow a^*$ in the mass-critical case $\alpha=\frac{4s}{d}$. Note that since $E(|u|) \leq E(u)$, we only need to consider non-negative minimizers for $I(a)$.

\begin{theorem} \label{theo-blowup}
	Let $d\geq 1$, $0<s<1$, $\alpha=\frac{4s}{d}$ and $V \in C(\R^d)$ satisfy \emph{(V1)} and \emph{(V2)}. Let $u_a$ be a non-negative minimizer for $I(a)$ with $a_*<a<a^*$ given in Theorem $\ref{theo-existence}$. Then $u_a$ blows up as $a \nearrow a^*$ in the sense that
	\begin{align} \label{blowup-meaning}
	\varep_a^{-s}:= \|(-\Delta)^{\frac{s}{2}} u_a\|_{L^2} \rightarrow \infty
	\end{align}
	as $a \nearrow a^*$. Moreover, up to a subsequence, there exist $(x_a)_{a \nearrow a^*} \subset [0,1]^d$ and $(z_a)_{a \nearrow a^*} \subset \Z^d$ such that $x_a \rightarrow x^0 \in [0,1]^d$ with $V(x^0)=\min V$ and
	\[
	\varep_a^{\frac{d}{2}} u_a(\varep_a \cdot + x_a + z_a) \rightarrow Q \text{ strongly in } H^s(\R^d) 
	\]
	as $a \nearrow a^*$, where $Q$ is the unique (up to translations) positive radial ground state related to \eqref{ell-equ-fra-cri}.
\end{theorem}

The proof of this result is based on the compactness of optimizing sequence for the Gagliardo-Nirenberg inequatlity (see Lemma $\ref{lem-compact-mini-fGN}$) which is another consequence of the concentration-compactness principle. 

Note that Theorem $\ref{theo-blowup}$ does not give any information on the blow-up rate of $\|(-\Delta)^{\frac{s}{2}} u_a\|_{L^2}$. Under an additional assumption on the external potential, we gave a detailed description of the blow-up behavior of minimizers for $I(a)$. More precisely, we have the following result.

\begin{theorem} \label{theo-blowup-refined}
	Let $d\geq 1$ and $0<s<1$. Let $\alpha=\frac{4s}{d}$. Let $V \in C(\R^d)$ satisfy \emph{(V1)} and \emph{(V2)}. Assume in addition that $V$ satisfies
	\begin{itemize}
		\item[(V3)] $\min V=0$, $V^{-1}(0)=x_0 + \Z^d$ for some $x_0 \in [0,1]^d$ and there exist $0<p<d+4s$ and $\kappa>0$ such that
		\[
		\lim_{x\rightarrow x_0} \frac{V(x)}{|x-x_0|^p} = \kappa.
		\]
	\end{itemize}
	Let $u_a$ be a non-negative minimizer for $I(a)$ with $a_*<a<a^*$ given in Theorem $\ref{theo-existence}$. Then up to a subsequence, there exist sequences $(x_a)_{a \nearrow a^*} \subset [0,1]^d$, $(z_a)_{a\nearrow a^*} \subset \Z^d$ such that $x_a \rightarrow x_0$ with
	\[
	\lim_{a \nearrow a^*} \frac{x_a-x_0}{\beta^{\frac{1}{2s+p}}_a} = 0
	\]
	and
	\[
	\beta_a^{\frac{d}{2(2s+p)}} u_a(\beta_a^{\frac{1}{2s+p}} \cdot + x_a + z_a) \rightarrow \lambda_0^{\frac{d}{2}} Q(\lambda_0 \cdot) \text{ strongly in } H^s(\R^d) \text{ as } a\nearrow a^*,
	\]
	where $Q$ is the unique (up to translations) positive radial ground state related to \eqref{ell-equ-fra-cri} and
	\begin{align} \label{blowup-scaling}
	\beta_a:= 1- \left(\frac{a}{a^*} \right)^\frac{2s}{d}, \quad \lambda_0 = \left( \frac{\kappa p}{d} \int_{\R^d} |x|^p [Q_0(x)]^2 dx \right)^{\frac{1}{2s+p}}, \quad Q_0= \frac{Q}{\|Q\|_{L^2}}.
	\end{align}
\end{theorem}
Theorem $\ref{theo-blowup-refined}$ gives a detailed description of the blow-up behavior of minimizers for $I(a)$ as $a \nearrow a^*$. In particular, the minimizer blows up at speed $\beta_a^{-\frac{1}{2s+p}}$ as $a \nearrow a^*$. The proof of Theorem $\ref{theo-blowup-refined}$ is based on energy estimates (see Lemma $\ref{lem-energy-est-cri}$). In \cite{DTWZ} and \cite{HL}, these energy estimates were obtained by analysing the corresponding Euler-Lagrange equation related to the minimizers. In this paper, we give an alternative approach based on the compactness of optimizing sequence for the fractional Gagliardo-Nirenberg inequality. Our argument is simpler and more direct than the ones in \cite{DTWZ, HL}. We finally point out that in constrast of the classical Schr\"odinger equation $s=1$ in which the ground state decays exponentially at infinity, the ground state for fractional Schr\"odinger equation $0<s<1$ decays only polynomially at infinity. This is the reason why we have to impose an additional condition $p<d+4s$ in (V3) which ensures that $|\cdot|^p [Q_0]^2$ is integrable.

The paper is organized as follows. In Section $\ref{section-preliminaries}$, we give some preliminaries related to our problem including the fractional Sobolev spaces, the fractional Gagliardo-Nirenberg inequality, the radial compactness embedding and the concentration-compactness principle. In Section $\ref{section-existence}$, we prove the existence and non-existence of minimizers in the case of no external potential and in the case of periodic potentials. Finally, we give the proof of blow-up behavior of minimizers in Section $\ref{section-blowup}$. 

\section{Preliminaries} \label{section-preliminaries}
\setcounter{equation}{0}

\subsection{Fractional Sobolev spaces}
For the reader's convenience, we recall the definition and some properties of the fractional Sobolev spaces. The fractional Laplacian $(-\Delta)^s$ is defined for $u \in \mathcal{S}(\R^d)$ by
\[
(-\Delta)^s u (x) := \mathcal{F}^{-1} \left( |\xi|^{2s} \mathcal{F} u(\xi) \right)(x),
\]
where $\mathcal{F}$ is the isometric Fourier transform in $L^2(\R^d)$
\[
\mathcal{F}(u) (\xi) = (2\pi)^{-\frac{d}{2}} \int_{\R^d} e^{-ix\cdot \xi} u(x) dx.
\] 
An alternative equivalent definition (see \cite{DPV}) is
\begin{align*}
(-\Delta)^s u(x) &= C(d,s) \lim_{\varep \searrow 0} \int_{\R^d \backslash B(0,\varep)} \frac{u(x) - u(y)}{|x-y|^{d+2s}} dy \\
&=-\frac{1}{2} C(d,s) \int_{\R^d} \frac{u(x+y) + u(x-y) - 2u(x)}{|y|^{d+2s}} dy,
\end{align*}
where
\begin{align} \label{defi-C-ds}
C(d,s) := \left(\int_{\R^d} \frac{1-\cos \zeta_1}{|\zeta|^{d+2s}} d\zeta \right)^{-1}.
\end{align}
For any $s \in (0,1)$, the fractional Sobolev space $H^s(\R^d)$ is defined by
\begin{align*}
H^s(\R^d) := \left\{ u \in L^2(\R^d) \ : \  (-\Delta)^{\frac{s}{2}} u \in L^2(\R^d)\right\}.
\end{align*}
The space $H^s(\R^d)$ is a Hilbert space endowed with the inner product
\[
\scal{u,v}_{H^s} := \mathfrak{Re} \int_{\R^d} u \overline{v} dx + \mathfrak{Re}\int_{\R^d} (-\Delta)^{\frac{s}{2}} u \overline{(-\Delta)^{\frac{s}{2}}v} dx
\]
and norm
\[
\|u\|^2_{H^s} = \|u\|^2_{L^2} + \|(-\Delta)^{\frac{s}{2}} u\|^2_{L^2}.
\]
Note that (\cite[Proposition 3.4]{DPV})
\begin{align} \label{equi-norm}
\|(-\Delta)^{\frac{s}{2}} u\|^2_{L^2} = \int_{\R^d} |\xi|^{2s} |\mathcal{F}u(\xi)|^2 d\xi = \frac{1}{2} C(d,s) \iint_{\R^{2d}} \frac{|u(x) - u(y)|^2}{|x-y|^{d+2s}} dx dy.
\end{align}

We have the following Sobolev embedding for fractional Sobolev spaces.
\begin{lemma} [Sobolev embedding \cite{BCD, DPV}] 
	Let $s \in (0,1)$. Then the following embeddings are continuous.
	\begin{itemize}
		\item $H^s(\R^d) \hookrightarrow L^q(\R^d)$, $2 \leq q \leq \frac{2d}{d-2s}$ if $d>2s$.
		\item $H^s(\R^d) \hookrightarrow L^q(\R^d)$, $2\leq q <\infty$ if $d=2s$.
		\item $H^s(\R) \hookrightarrow C_b(\R)$ if $1<2s$, where $C_b(\R)$ is the space of continuous bounded functions on $\R$. 
	\end{itemize}
	Moreover, $H^s(\R^d) \hookrightarrow L^q(\R^d)$ is locally compact for any $1\leq q <2^*$, where
	\begin{align} \label{defi-2-sup}
	2^* := \left\{
	\begin{array}{cl}
	\frac{2d}{d-2s} &\text{if } d>2s, \\
	\infty &\text{if } d \leq 2s.
	\end{array}
	\right.
	\end{align}
\end{lemma}

\begin{lemma} [\cite{BCD}]
	Let $s\in (0,1)$. The space $\mathcal{S}(\R^d)$ is dense in $H^s(\R^d)$. Moreover, the multiplication by a Schwarz function is a continuous map from $H^s(\R^d)$ to itself.
\end{lemma}

\subsection{Fractional Gagliardo-Nirenberg inequality}
Let us recall the following fractional Gagliardo-Nirenberg inequality
\begin{align} \label{fractional-GN}
\|u\|^{\alpha+2}_{L^{\alpha+2}} \leq C_{\text{opt}} \|(-\Delta)^{\frac{s}{2}} u\|^{\frac{d\alpha}{2s}}_{L^2} \|u\|^{\alpha+2 - \frac{d\alpha}{2s}}_{L^2}.
\end{align}
The sharp constant $C_{\text{opt}}>0$ can be obtained by
\[
\frac{1}{C_{\text{opt}}} = \inf\left\{J(u) \ : \ u \in H^s(\R^d) \backslash \{0\}\right\},
\]
where
\[
J(u):= \left[ \|(-\Delta)^{\frac{s}{2}} u\|^{\frac{d\alpha}{2s}}_{L^2} \|u\|^{\alpha+2 - \frac{d\alpha}{2s}}_{L^2} \right] \div \|u\|^{\alpha+2}_{L^{\alpha+2}}
\]
is the Weinstein functional. We have the following result due to \cite{DTWZ, FL, FLS}.
\begin{theorem} [\cite{DTWZ, FL, FLS}] \label{theo-fractional-GN}
	Let $d\geq 1$, $0<s<1$ and $0<\alpha<s^*$. Then the fractional Gagliardo-Nirenberg inequality \eqref{fractional-GN} is attained at a function $Q_\alpha \in H^s(\R^d)$ with the following properties:
	\begin{itemize}
		\item $Q_\alpha(x)$ is radially symmetric, positive and strictly decreasing in $|x|$.
		\item $Q_\alpha$ belongs to $H^{2s+1}(\R^d) \cap C^\infty(\R^d)$ and satisfies
		\begin{align}
		\frac{C_1}{1+|x|^{d+2s}} \leq Q_\alpha(x) &\leq \frac{C_2}{1+|x|^{d+2s}}, \label{decay-1} \\
		|\partial_j Q_\alpha(x)| &\leq \frac{C_3}{1+|x|^{d+2s}}, \quad j=1, \cdots, d, \label{decay-2}
		\end{align}
		for all $x \in \R^d$, where $C_1, C_2$ and $C_3$ are positive constants depending on $s, d, \alpha$ and $Q_\alpha$.
		\item $Q_\alpha$ solves the elliptic equation \eqref{ell-equ-fra}.
	\end{itemize}
	Moreover, every minimizer $\varphi \in H^s(\R^d)$ for the fractional Gagliardo-Nirenberg inequality \eqref{fractional-GN} is of the form $\varphi(x) = \beta Q_\alpha(\gamma(x+y))$ with some $\beta \in \C, \beta \ne 0$, $\gamma>0$ and $y \in \R^d$.
\end{theorem}
Note that \eqref{decay-2} follows from \eqref{decay-1} and Lemma C.2 in \cite{FLS}.
\begin{remark} 
	We have the following Pohozaev's identities related to \eqref{ell-equ-fra}
	\begin{align} \label{pohozaev-identity-alpha}
	\|(-\Delta)^{\frac{s}{2}} Q_\alpha\|^2_{L^2} = \frac{d\alpha}{2s(\alpha+2)} \|Q_\alpha\|^{\alpha+2}_{L^{\alpha+2}} = \frac{d\alpha}{4s-(d-2s)\alpha} \|Q_\alpha\|^2_{L^2}.
	\end{align}
	In particular, 
	\begin{align*}
	C_{\text{opt}} &= \|Q_\alpha\|^{\alpha+2}_{L^{\alpha+2}} \div \left[ \|(-\Delta)^{\frac{s}{2}} Q_\alpha\|^{\frac{d\alpha}{2s}}_{L^2} \|Q_\alpha\|^{\alpha+2-\frac{d\alpha}{2s}}_{L^2} \right] \\
	&= \left(\frac{2s(\alpha+2)-d\alpha}{d\alpha} \right)^{\frac{d\alpha}{4s}} \frac{2s(\alpha+2)}{2s(\alpha+2)- d\alpha} \frac{1}{\|Q_\alpha\|^\alpha_{L^2}}.
	\end{align*}
	In the case $\alpha =\frac{4s}{d}$, we denote $Q:=Q_{\frac{4s}{d}}$ and have that
	\begin{align} \label{pohozaev-identity}
	\|(-\Delta)^{\frac{s}{2}} Q\|^2_{L^2} = \frac{d}{d+2s} \|Q\|^{\frac{4s}{d} +2}_{L^{\frac{4s}{d}+2}} = \frac{d}{2s} \|Q\|^2_{L^2}
	\end{align}
	and
	\[
	C_{\text{opt}} = \frac{d+2s}{d} \frac{1}{\|Q\|^{\frac{4s}{d}}_{L^2}}.
	\]
\end{remark}

\subsection{Radial compactness embedding}
Denote
\[
H^s_{\text{r}}(\R^d):= \{ u\in H^s(\R^d) \ : \ u \text{ is radially symmetric}\}.
\]
We have the following estimate for radially symmetric functions in $H^s(\R^d)$. 
\begin{lemma} [\cite{CO}] 
	Let $d\geq 2$ and $\frac{1}{2}<s<\frac{d}{2}$. Then every radially symmetric function $u \in H^s(\R^d)$ is almost everywhere equal to a function $U$, continuous for $x \ne 0$ such that
	\[
	|U(x)| \leq C |x|^{s-\frac{d}{2}} \|(-\Delta)^{\frac{s}{2}} u\|_{L^2}, 
	\]
	for all $x\ne 0$, where $C>0$ depends only on $d$ and $s$.
\end{lemma}
Using this radial estimate, we obtain the following radial compactness embedding for fractional Sobolev spaces.
\begin{lemma}
	Let $d\geq 2$ and $\frac{1}{2}<s<\frac{d}{2}$. The injection $H^s_{\emph{r}}(\R^d) \hookrightarrow L^q(\R^d)$ is compact for any $2<q<\frac{2d}{d-2s}$.
\end{lemma}
If we assume that $u \in H^s_{\text{r}}(\R^d)$ is radially decreasing, i.e. $u(x) \leq u(y)$ if $|x| \geq |y|$, then we have  the following compact embedding.
\begin{lemma} 
	Let $d\geq 1$ and $s\in (0,1)$. The injection $H^s_{\emph{rd}}(\R^d) \hookrightarrow L^q(\R^d)$ is compact for any $2<q<2^*$, where
	\[
	H^s_{\emph{rd}}(\R^d):= \left\{ u \in H^s_{\emph{r}}(\R^d) \ : \ u \text{ is radially decreasing}\right\}.
	\]
\end{lemma}
\begin{proof}
	It is enough to show that if $(u_n)_{n\geq 1} \subset H^s_{\text{rd}}$ is such that $u_n \rightharpoonup 0$ weakly in $H^s(\R^d)$, then $u_n \rightarrow 0$ strongly in $L^q(\R^d)$. We can assume that $u_n \rightarrow 0$ strongly in $L^q_{\text{loc}}(\R^d)$ and $u_n \rightarrow 0$ almost everywhere. It follows that
	\[
	\|u_n\|^q_{L^q} = \|u_n\|^q_{L^q(B)} + \|u_n\|^q_{L^q(B^c)} = \|u_n\|^q_{L^q(B^c)} + o_n(1),
	\]
	where $B$ is the unit ball in $\R^d$ and $B^c$ is its complement in $\R^d$. Here $o_n(1)$ means that $o_n(1) \rightarrow 0$ as $n\rightarrow \infty$. It remains to estimate the term on $B^c$. To do this, we use the fact that $u$ is radially symmetric and radially decreasing to have that
	\[
	\|u\|^2_{L^2} = \int |u(x)|^2 dx \geq |\mathbb{S}^{d-1}| \int_0^r |u(\tau)|^2 \tau^{d-1} d\tau \geq |\mathbb{S}^{d-1}| |u(r)|^2\frac{r^d}{d}.
	\]
	It follows that
	\[
	|u(x)| \leq \left(\frac{d}{|\mathbb{S}^{d-1}|} \right)^{\frac{1}{2}} |x|^{-\frac{d}{2}}\|u\|_{L^2}
	\]
	for all $x \ne 0$. Using this inequality and the fact $(u_n)_{n \geq 1}$ is bounded in $H^s(\R^d)$, we infer that there exists $C(d)>0$ independent of $n$ such that
	\[
	|u_n(x)|^q \leq C(d) |x|^{-\frac{dq}{2}}.
	\]
	Note that the last term is integrable on $B^c$ since $q>2$. It follows from the dominated convergence theorem that $\|u_n\|^q_{L^q(B^c)} \rightarrow 0$ as $n\rightarrow \infty$. The proof is complete.
\end{proof}

We next recall the symmetric rearrangement which is useful in variational calculus. The symmetric rearrangement of a measurable function $u: \R^d \rightarrow \C$  vanishing at infinity is defined by
\[
u^\ast (x) = \int_0^\infty \chi_{\{|u|>t\}^\ast}(x) dt,
\]  
where $\{|u|>t\}^\ast$ is a ball centered at the origin whose volume equals to the volume of $\{x \in \R^d \ : \ |u(x)|>t\}$ and $\chi_B$ is the characteristic function of $B$. We see that $u^\ast$ is non-negative, radially symmetric and radially decreasing. We recall some basic properties of the symmetric rearrangement:
\begin{itemize}
	\item $L^q$-norm preserving: $\|u\|_{L^q} = \|u^\ast\|_{L^q}$ for any $1\leq q \leq \infty$.
	\item Order-preserving: if $u(x) \leq v(x)$ for all $x\in \R^d$, then $u^\ast(x) \leq v^\ast(x)$ for all $x\in \R^d$.
	\item Hardy-Littlewood inequality: if $u,v$ are non-negative measurable functions that vanish at infinity, then
	\[
	\int_{\R^d} u(x) v(x) dx \leq \int_{\R^d} u^\ast(x) v^\ast(x) dx,
	\]
	in the sense that the left hand side is finite whenever the right hand side is finite.
	\item Decreasing $L^q$-distance: $\|u-v\|_{L^q} \geq \|u^\ast - v^\ast\|_{L^q}$ for any $1\leq q \leq \infty$.
	\item Polya-Szeg\"o inequality: $\|\nabla u^\ast\|_{L^q} \leq \|\nabla u\|_{L^q}$ for any $1\leq q\leq \infty$. In particular, 
	\[
	\|\nabla |u|^\ast\|_{L^q} \leq \|\nabla |u|\|_{L^q} \leq \|\nabla u\|_{L^q}, \quad 1 \leq q\leq \infty.
	\]
	The last inequality follows from the diamagnetic inequality.
	\item Fractional Polya-Szeg\"o inequality \cite{Park}: for $0<s<1$, $\|(-\Delta)^{\frac{s}{2}} u^\ast\|_{L^2} \leq \|(-\Delta)^{\frac{s}{2}} u\|_{L^2}$. The equality holds for radially decreasing function. In particular,
	\[
	\|(-\Delta)^{\frac{s}{2}} |u|^\ast\|_{L^2} \leq \|(-\Delta)^{\frac{s}{2}} |u|\|_{L^2} \leq \|(-\Delta)^{\frac{s}{2}} u\|_{L^2}.
	\]
	Here the last inequality follows from \eqref{equi-norm} and the fact $||u(x)| - |u(y)|| \leq |u(x) - u(y)|$.
\end{itemize}

\subsection{Concentration-compactness principle}

\begin{lemma}[Concentration-compactness lemma] \label{lem-concen-compact}
Let $d\geq 1$ and $s \in (0,1)$. Let $(u_n)_{n\geq 1}$ be a bounded sequence in $H^s(\R^d)$ satisfying 
\begin{align} \label{limit-un}
\lim_{n\rightarrow \infty} \|u_n\|^2_{L^2} =a
\end{align}
for some fixed constant $a>0$. Then there exists a subsequence $(u_{n_k})_{k\geq 1}$ satisfying one of the following three possibilities.
\begin{itemize}
\item {\bf Vanishing:} 
\[
\lim_{k\rightarrow \infty} \sup_{y\in \R^d} \int_{B(y,R)} |u_{n_k}(x)|^2 dx =0
\]
for all $R>0$.
\item {\bf Compactness:} There exists a sequence $(y_k)_{k\geq 1} \subset \R^d$ such that for all $\varep>0$, there exist $R(\varep)>0$ and $k(\varep) \geq 1$ such that for all $k\geq k(\varep)$,
\[
\int_{B(y_k,R(\varep))} |u_{n_k}(x)|^2 dx \geq a-\varep.
\]
\item {\bf Dichotomy:} There exist $\mu \in (0,a)$ and sequences $(u^1_k)_{k\geq 1}, (u^2_k)_{k\geq 1}$ bounded in $H^s(\R^d)$ such that
\begin{align} \label{dichotomy-possibility}
\left\{
\renewcommand{\arraystretch}{1.2}
\begin{array}{l} 
\|u_{n_k} - u^1_k - u^2_k\|_{L^q} \rightarrow 0 \text{ as } k\rightarrow \infty \text{ for all } 2\leq q<2^*, \\
\|u^1_k\|^2_{L^2} \rightarrow \mu, \quad \|u^2_k\|_{L^2} \rightarrow a-\mu  \text{ as } k\rightarrow \infty,\\
\emph{dist}(\emph{supp}(u^1_k), \emph{supp}(u^2_k)) \rightarrow \infty \text{ as } k \rightarrow \infty,\\
\liminf_{k\rightarrow \infty} \|(-\Delta)^{\frac{s}{2}} u_{n_k}\|^2_{L^2} - \|(-\Delta)^{\frac{s}{2}} u^1_{k}\|^2_{L^2} -\|(-\Delta)^{\frac{s}{2}} u^2_{k}\|^2_{L^2} \geq 0.
\end{array}
\right.
\end{align}
\end{itemize}
\end{lemma}

\begin{proof} The proof of this result follows by the same argument as in \cite{Lions1, Lions2}. For the reader's convenience, we give some details. Define the so-called L\'evy concentration functions $M_n: [0,\infty) \rightarrow [0,a]$ by
\[
M_n(R):= \sup_{y\in \R^d} \int_{B(y,R)} |u_n(x)|^2 dx.
\]
Since $\|u_n\|^2_{L^2} =a$ for all $n\geq 1$, we see that $(M_n)_{n\geq 1}$ is a uniformly bounded sequence of non-decreasing, non-negative functions on $[0,\infty)$. By Helly's collection theorem, there exist a subsequence $(M_{n_k})_{k\geq 1}$ and a non-negative, non-decreasing function $M: [0,\infty) \rightarrow [0,a]$ such that $\lim_{k\rightarrow \infty} M_{n_k}(R) = M(R)$ for any $R\geq 0$. Set
\[
\mu := \lim_{R\rightarrow \infty} M(R) = \lim_{R\rightarrow \infty} \lim_{k\rightarrow \infty} M_{n_k}(R).
\]
Obviously, $\mu \in [0,a]$. We will consider three possibilities: $\mu=0$, $\mu =a$ and $\mu \in (0,a)$.

If $\mu=0$, then since $M$ is a non-negative, non-decreasing function it follows that $M(R) =0$ for all $R\geq 0$, equivalently, $\lim_{k\rightarrow \infty} M_{n_k}(R) =0$ for all $R\geq 0$.

If $\mu=a$, we proceed as follows. We first find $R_0>0$ large enough such that $M(R_0) > \frac{a}{2}$. For any $k\geq 1$, there exists $y_k \in \R^d$ such that
\[
M_{n_k}(R_0) \leq \int_{B(y_k, R_0)} |u_{n_k}(x)|^2 dx +\frac{1}{k}.
\]
Now, for $0<\varep <\frac{a}{2}$, we fix $\tau(\varep)>0$ such that $M(\tau(\varep)) >a-\varep>\frac{a}{2}$ and let $x_k \in \R^d$ satisfying
\[
M_{n_k}(\tau(\varep)) \leq \int_{B(x_k, \tau(\varep))} |u_{n_k}(x)|^2 dx +\frac{1}{k}.
\]
We thus have that
\begin{align*}
\int_{B(y_k,R_0)} |u_{n_k}(x)|^2 dx + \int_{B(x_k, R(\varep))} |u_{n_k}(x)|^2 dx &\geq M_{n_k}(R_0) + M_{n_k}(\tau(\varep)) +o_k(1) \\
&\geq M(R_0) + M(\tau(\varep)) + o_k(1),
\end{align*}
where $o_k(1)$  means that $o_k(1) \rightarrow 0$ as $k\rightarrow \infty$. Since $M(R_0) + M(\tau(\varep)) >a$, we infer that there exists $k(\varep)$ large enough such that the right hand side is strictly larger than $a$ for all $k\geq k(\varep)$. This implies that for $k\geq k(\varep)$,
\[
B(y_k, R_0) \cap B(x_k, \tau(\varep)) \ne \emptyset,
\]
that is, $B(x_k, \tau(\varep)) \subset B(y_k, R_0(\varep))$ with $R_0(\varep) = R_0 + 2\tau(\varep)$. Hence, by enlarging $k(\varep)$ if necessary,
\begin{align} \label{estimate-limit}
\int_{B(y_k, R_0(\varep))} |u_{n_k}(x)|^2 dx \geq \int_{B(x_k, \tau(\varep))} |u_{n_k}(x)|^2 dx \geq a-2\varep
\end{align}
for all $k\geq k(\varep)$ which shows the compactness with $R(\varep)=R_0(\varep)$.

We now consider the case $\mu \in (0,a)$. Take $\varphi, \psi \in C^\infty_0(\R^d)$ satisfying $0\leq \varphi, \psi \leq 1$ and
\[
\varphi(x) = \left\{
\begin{array}{cl}
1 &\text{if } |x|\leq 1, \\
0 &\text{if } |x| \geq 2,
\end{array}
\right.
\quad \text{and} \quad 
\psi(x) = \left\{
\begin{array}{cl}
0 &\text{if } |x|\leq 1, \\
1 &\text{if } |x| \geq 2.
\end{array}
\right.
\]
Denote for $R>0$ the functions $\varphi_R(x)= \varphi(x/R)$ and $\psi_R(x)= \psi(x/R)$. 
We first note that for $v \in H^s(\R^d)$,
\begin{align*}
2[C(d,s)]^{-1}\|(-\Delta)^{\frac{s}{2}}(\varphi_R v)\|^2_{L^2} &= \iint \frac{|\varphi_R(x) v(x) - \varphi_R(y) v(y)|^2}{|x-y|^{d+2s}} dx dy \\
&= \iint \frac{[\varphi_R(x)]^2 |v(x) - v(y)|^2}{|x-y|^{d+2s}} dx dy + \iint \frac{|\varphi_R(x) - \varphi_R(y)|^2|v(y)|^2}{|x-y|^{d+2s}} dx dy \\
&\mathrel{\phantom{=}} + 2 \mathfrak{Re}\iint \frac{\varphi_R(x) (v(x)-v(y)) (\varphi_R(x)-\varphi_R(y)) \overline{v}(y)}{|x-y|^{d+2s}} dx dy \\
&= \iint \frac{[\varphi_R(x)]^2 |v(x) - v(y)|^2}{|x-y|^{d+2s}} dx dy + \text{I} + \text{II},
\end{align*}
where $C(d,s)$ is given in \eqref{defi-C-ds}. By H\"older's inequality,
\begin{align*}
|\text{II}| &\leq 2 \left(\iint \frac{[\varphi_R(x)]^2|v(x)-v(y)|^2}{|x-y|^{d+2s}} dx dy \right)^{\frac{1}{2}} \left(\iint \frac{|\varphi_R(x) - \varphi_R(y)|^2|v(y)|^2}{|x-y|^{d+2s}} dx dy \right)^{\frac{1}{2}} \\
&\leq 2 \|(-\Delta)^{\frac{s}{2}} v\|_{L^2} \left(\iint \frac{|\varphi_R(x) - \varphi_R(y)|^2|v(y)|^2}{|x-y|^{d+2s}} dx dy \right)^{\frac{1}{2}},
\end{align*}
where we have used that $0\leq \varphi_R \leq 1$. We thus only consider $\text{I}$. Using the fact $|\nabla\varphi_R| \lesssim R^{-1}$, we have that
\begin{align*}
\text{I} &= \iint_{|x-y| \leq R} \frac{|\varphi_R(x) - \varphi_R(y)|^2|v(y)|^2}{|x-y|^{d+2s}} dx dy  + \iint_{|x-y| \geq R} \frac{|\varphi_R(x) - \varphi_R(y)|^2|v(y)|^2}{|x-y|^{d+2s}} dx dy \\
&\lesssim R^{-2}\iint_{|x-y|\leq R} \frac{|v(y)|^2}{|x-y|^{d+2s-2}} dx dy + \iint_{|x-y|\geq R} \frac{|v(y)|^2}{|x-y|^{d+2s}} dx dy \\
&= R^{-2} \int |v(y)|^2 dy \int_{|\zeta| \leq R} |\zeta|^{-(d+2s-2)} d\zeta + \int |v(y)|^2 dx \int_{|\zeta| \geq R} |\zeta|^{-(d+2s)} d\zeta \\
&= C R^{-2s} \|v\|^2_{L^2}.
\end{align*}
We thus prove that
\begin{align} \label{product-est}
\|(-\Delta)^{\frac{s}{2}} (\varphi_R v)\|^2_{L^2} - \frac{1}{2} C(d,s)\iint \frac{[\varphi_R(x)]^2 |v(x) - v(y)|^2}{|x-y|^{d+2s}} dx dy \leq C R^{-s} \|v\|^2_{H^s}.
\end{align}
In particular, since $(u_n)_{n\geq 1}$ is a bounded sequence in $H^s(\R^d)$, we see that
\begin{align} \label{product-est-app}
\|(-\Delta)^{\frac{s}{2}}(\varphi_R u_n)\|^2_{L^2} - \frac{1}{2} C(d,s)\iint \frac{[\varphi_R(x)]^2 |u_n(x) - u_n(y)|^2}{|x-y|^{d+2s}} dx dy \leq CM^2 R^{-s},
\end{align}
where $M:= \sup_{n\geq 1} \|u_n\|_{H^s}$. A similar estimate holds for $\psi_R$ in place of $\varphi_R$. 

Now let $\varep>0$. We choose $R_0>0$ large enough such that $M(R) \geq \mu - \frac{\varep}{4}$ for all $R \geq R_0$. We next choose $R_1>0$ large enough so that $CM^2 R_1^{-s} \leq \varep$. Of course, we may assume $R_1 \geq R_0$ and thus $M(R_1)\geq \mu -\frac{\varep}{4}$. Since $\lim_{k\rightarrow \infty} M_{n_k}(R_1) = M(R_1)$, there exists $k_1 \geq 1$ large enough such that for $k\geq k_1$,
\[
\mu-\frac{\varep}{2} < M_{n_k}(R_1) \leq M_{n_k}(2R_1) < \mu +\varep.
\]
On the other hand, for each $k\geq 1$, there exists $y_k \in \R^d$ such that
\[
M_{n_k}(R_1) \leq \int_{B(y_k,R_1)} |u_{n_k}(x)|^2 dx +\frac{1}{k}.
\]
By enlarging $k_1$ if necessary, we have for $k\geq k_1$,
\[
\mu- \varep <\int_{B(y_k,R_1)} |u_{n_k}(x)|^2 dx \leq M_{n_k}(R_1) \leq M_{n_k}(2R_1) <\mu+\varep.
\]
We now set $u^1_k(x) = \varphi_{R_1}(x-y_k) u_{n_k}$. It follows that for $k\geq k_1$,
\begin{align*}
\|u^1_k\|^2_{L^2} - \mu &\geq \int_{B(y_k, R_1)} |u_{n_k}(x)|^2 dx - \mu > -\varep, \\
\|u^1_k\|^2_{L^2} - \mu &\leq \int_{B(y_k,2R_1)} |u_{n_k}(x)|^2 dx - \mu \leq M_{n_k}(2R_1) -\mu <\varep,
\end{align*}
which imply that $\lim_{k\rightarrow \infty}\|u^1_k\|^2_{L^2} = \mu$. Moreover, by the choice of $R_1$ and \eqref{product-est-app}, we have for $k\geq k_1$,
\begin{align} \label{est-u-1k}
\|(-\Delta)^{\frac{s}{2}} u^1_k\|^2_{L^2} - \frac{1}{2} C(d,s)\iint \frac{[\varphi_{R_1}(x-y_k)]^2 |u_{n_k}(x) - u_{n_k}(y)|^2}{|x-y|^{d+2s}} dx dy \leq CM^2 R_1^{-s} \leq \varep.
\end{align}
We next set $u^2_k(x) = \psi_{R_k}(x-y_k) u_{n_k}$, where $R_k \rightarrow \infty$ as $k\rightarrow \infty$ is such that $M_{n_k}(2R_k) \leq \mu+\varep$. Observe that
\begin{align*}
\|u_{n_k} - u^1_k - u^2_k\|^2_{L^2} &= \int \left(1-\varphi^2_{R_1}(x-y_k) - \psi_{R_k}^2(x-y_k)\right) |u_{n_k}(x)|^2 dx \\
&\leq \int_{R_1 \leq |x-y_k|\leq 2R_k} |u_{n_k}(x)|^2 dx \\
&= \int_{B(y_k,2R_k)} |u_{n_k}(x)|^2 dx - \int_{B(y_k, R_1)} |u_{n_k}(x)|^2 dx \\
&\leq M_{n_k}(2R_k) - (\mu-\varep) \leq 2\varep.
\end{align*}
Since $\varep>0$ is arbitrary, it follows that 
\begin{align} \label{L2-limit}
\|u_{n_k} - u^1_k - u^2_k\|^2_{L^2} \rightarrow 0
\end{align}
as $k\rightarrow \infty$, hence $\lim_{k\rightarrow \infty} \|u^2_k\|^2_{L^2} = a- \mu$ and also for $k$ large enough
\begin{align} \label{est-u-2k}
\|(-\Delta)^{\frac{s}{2}} u^2_k\|^2_{L^2} - \frac{1}{2} C(d,s)\iint \frac{[\psi_{R_k}(x-y_k)]^2 |u_{n_k}(x) - u_{n_k}(y)|^2}{|x-y|^{d+2s}} dx dy \leq CM^2 R_k^{-s} 
\leq \varep.
\end{align}
By using the fractional Gagliardo-Nirenberg inequality, we have for $2 \leq q <2^*$,
\[
\|u_{n_k} - u^1_k - u^2_k\|_{L^q} \leq C\|(-\Delta)^{\frac{s}{2}} (u_{n_k} - u^1_k - u^2_k)\|^{\frac{d}{2s} -\frac{d}{sq}}_{L^2} \|u_{n_k} - u^1_k - u^2_k\|^{1-\frac{d}{2s} +\frac{d}{sq}}_{L^2}.
\]
Since the multiplication by a Schwarz function is a continuous map from $H^s(\R^d)$ to itself, we see that $(u^1_k)_{k\geq 1}$ and $(u^2_k)_{k\geq 1}$ are bounded in $H^s(\R^d)$. It follows from \eqref{L2-limit} that
\[
\|u_{n_k} - u^1_k - u^2_k\|_{L^q} \rightarrow 0
\]
as $k\rightarrow \infty$ for any $2\leq q<2^*$. Moreover, by construction
\[
\text{dist}\left( \text{supp}(u^1_k), \text{supp}(u^2_k)\right) \geq 2R_k - R_1 \rightarrow \infty
\]
as $k\rightarrow \infty$. Finally, the last conclusion follows from \eqref{est-u-1k}, \eqref{est-u-2k} and the fact
\[
1 - [\varphi_{R_1}(x-y_k)]^2 - [\psi_{R_k}(x-y_k)]^2 \geq 0.
\] 
The proof is complete.
\end{proof}

\begin{remark} \label{rem-concen-compact-all-k}
	If instead of \eqref{limit-un}, we assume $\|u_n\|^2_{L^2}=a$ for all $n\geq 1$, then the compactness holds for all $k\geq 1$, i.e. there exists a sequence $(y_k)_{k\geq 1} \subset \R^d$ such that for all $\varep>0$, there exists $R(\varep)>0$ such that for all $k\geq 1$,
	\[
	\int_{B(y_k,R(\varep))} |u_{n_k}(x)|^2 dx \geq a-\varep.
	\]
	In fact, since $\|u_{n_k}\|^2_{L^2} =a$ for all $k\geq 1$, we see that for each $1\leq k\leq k(\varep)$, there exists $R_k(\varep)>0$ such that
	\[
	\int_{B(y_k, R_k(\varep))} |u_{n_k}(x)|^2 dx \geq a-\varep,
	\]
	where $k(\varep)$ is given in \eqref{estimate-limit}. Taking $R(\varep) := \max \{ R_0(\varep), R_1(\varep), \cdots, R_{k(\varep)}(\varep)\}$ with $R_0(\varep)$ as in \eqref{estimate-limit}, we obtain
	\[
	\int_{B(y_k,R(\varep))} |u_{n_k}(x)|^2 dx \geq a-\varep
	\]
	for all $k\geq 1$.
\end{remark}

\begin{remark} \label{rem-concen-compact}
	\begin{itemize}
		\item If the vanishing occurs, then $(u_{n_k})_{k\geq 1}$ converges strongly to 0 in $L^q(\R^d)$ for any $2<q<2^*$. Indeed, it follows from the following fact (see e.g. \cite{FQT}): if $(f_n)_{n\geq 1}$ is a bounded sequence in $H^s(\R^d)$ and satisfies
		\[
		\lim_{n\rightarrow \infty} \sup_{y \in \R^d} \int_{B(y,R)} |f_n(x)|^2 dx = 0
		\]
		for some $R>0$, then $f_n \rightarrow 0$ in $L^q(\R^d)$ for any $2<q<2^*$. 
		\item If the compactness occurs, then we infer that up to a subsequence, $u_{n_k}(\cdot+y_k)$ converges strongly to $u$ in $L^q(\R^d)$ for any $2\leq q<2^*$. Indeed, there exists a sequence $(y_k)_{k\geq 1} \subset \R^d$ such that for each $l \geq 1$, there exist $R_l>0$ and $k_l\geq 1$ such that for all $k\geq k_l$,
		\begin{align} \label{compactness-app}
		\int_{B(y_k,R_l)} |u_{n_k}(x)|^2 dx \geq a - \frac{1}{l} \quad \text{or} \quad \int_{B(0,R_l)} |\tilde{u}_{n_k}(x)|^2 dx \geq a- \frac{1}{l},
		\end{align}
		where $\tilde{u}_{n_k} (x):= u_{n_k}(x+y_k)$. Since the translated sequence $(\tilde{u}_{n_k})_{k\geq 1}$ is bounded in $H^s(\R^d)$, so up to subsequence, $\tilde{u}_{n_k} \rightharpoonup u$ weakly in $H^s(\R^d)$. By the lower semi-continuity of weak convergence, $\|u\|^2_{L^2} \leq a$. For each $l\geq 1$, the embedding $H^s(B(0,R_l)) \hookrightarrow L^2(B(0,R_l))$ is compact, so up to a subsequence, we have $\tilde{u}_{n_k} \rightarrow u$ strongly in $L^2(B(0,R_l))$. By a standard diagonalization argument, one may assume that there exists a subsequence still denoted by $(\tilde{u}_{n_k})_{k\geq 1}$ satisfies
		\[
		\tilde{u}_{n_k} \rightarrow u \text{ strongly in } L^2(B(0,R_l))
		\]
		for every $l\geq 1$. Taking $k\rightarrow \infty$ in \eqref{compactness-app}, we obtain
		\[
		a- \frac{1}{l} \leq \lim_{k\rightarrow \infty} \int_{B(0,R_l)} |\tilde{u}_{n_k}(x)|^2 dx = \|u\|^2_{L^2(B(0,R_l))} \leq \|u\|^2_{L^2}.
		\]
		Since $l\geq 1$ is arbitrary, it follows that $\|u\|^2_{L^2} =a$. Thus $\tilde{u}_{n_k} \rightarrow u$ strongly in $L^2(\R^d)$. By interpolating between $L^2(\R^d)$ and $L^{2^*}(\R^d)$ and using Sobolev embedding together with the boundedness of $(u_{n_k})_{k\geq 1}$ in $H^s(\R^d)$, we prove that $\tilde{u}_{n_k} \rightarrow u$ strongly in $L^q(\R^d)$ for any $2\leq q<2^*$. 
	\end{itemize}
\end{remark}

A direct application of the concentration-compactness lemma is the following compactness of minimizing sequence for the fractional Gagliardo-Nirenberg inequality.
\begin{lemma} \label{lem-compact-mini-fGN}
	Let $(u_n)_{n\geq 1}$ be a non-negative bounded sequence in $H^s(\R^d)$ satisfying
	\begin{align} \label{assumption-compact-mini-fGN-1}
	\lim_{n\rightarrow \infty} \|u_n\|^2_{L^2} = a^*, \quad \lim_{n\rightarrow \infty} \|(-\Delta)^{\frac{s}{2}} u_n\|^2_{L^2} - \frac{d}{d+2s} \|u_n\|^{\frac{4s}{d}+2}_{L^{\frac{4s}{d}+2}}  =0
	\end{align}
	and
	\begin{align} \label{assumption-compact-mini-fGN-2}
	\liminf_{n\rightarrow \infty} \|u_n\|^{\frac{4s}{d}+2}_{L^{\frac{4s}{d}+2}} >0,
	\end{align}
	where $a^*$ is defined in \eqref{defi-a*}. Then there exist a subsequence $(u_{n_k})_{k\geq 1}$ and a sequence $(y_k)_{k\geq 1} \subset \R^d$ such that
	\[
	u_{n_k}(\cdot + y_k) \rightarrow \gamma^{\frac{d}{2}} Q(\gamma \cdot) \text{ strongly in } H^s(\R^d)
	\]	
	for some $\gamma>0$ as $k\rightarrow \infty$.
\end{lemma}

\begin{proof}
	By Lemma $\ref{lem-concen-compact}$, there exists a subsequence $(u_{n_k})_{k\geq 1}$ satisfying one of the following three possibilities: vanishing, dichotomy and compactness.
	
	\noindent {\bf No vanishing.} Suppose that the vanishing occurs. We have from Remark $\ref{rem-concen-compact}$ that $u_{n_k} \rightarrow 0$ strongly in $L^q$ for any $2<q<2^*$. In particular, $\|u_{n_k}\|^{\frac{4s}{d}+2}_{L^{\frac{4s}{d}+2}} \rightarrow 0$ which contradicts to \eqref{assumption-compact-mini-fGN-2}. 
	
	\noindent {\bf No dichotomy.} If the dichotomy occurs, then there exist $\mu \in (0, a^*)$ with $a^*:= \|Q\|^2_{L^2}$ and sequences $(u^1_k)_{k\geq 1}, (u^2_k)_{k\geq 1}$ bounded in $H^s(\R^d)$ such that \eqref{dichotomy-possibility} holds. We infer that
	\begin{align} \label{compact-mini-fGN-proof-1}
	\|u_{n_k}\|^{\frac{4s}{d}+2}_{L^{\frac{4s}{d}+2}} = \|u^1_k\|^{\frac{4s}{d}+2}_{L^{\frac{4s}{d}+2}} + \|u^2_k\|^{\frac{4s}{d}+2}_{L^{\frac{4s}{d}+2}} +o_k(1).
	\end{align}
	By the fractional Gagliardo-Nirenberg inequality,
	\begin{align*}
	\|(-\Delta)^{\frac{s}{2}} u^1_k\|^2_{L^2} &\geq \frac{d}{d+2s} \left( \frac{a^*}{\|u^1_k\|^2_{L^2}} \right)^{\frac{2s}{d}} \|u^1_k\|^{\frac{4s}{d}+2}_{L^{\frac{4s}{d}+2}} \\
	&= \frac{d}{d+2s} \left( \frac{a^*}{\mu} \right)^{\frac{2s}{d}} \|u^1_k\|^{\frac{4s}{d}+2}_{L^{\frac{4s}{d}+2}} + o_k(1).
	\end{align*}
	Similarly,
	\[
	\|(-\Delta)^{\frac{s}{2}} u^2_k\|^2_{L^2}  \geq \frac{d}{d+2s} \left( \frac{a^*}{a^*-\mu} \right)^{\frac{2s}{d}} \|u^2_k\|^{\frac{4s}{d}+2}_{L^{\frac{4s}{d}+2}} + o_k(1).
	\]
	It follows from \eqref{dichotomy-possibility} that
	\begin{align*}
	\|(-\Delta)^{\frac{s}{2}} u_{n_k}\|^2_{L^2} &\geq \|(-\Delta)^{\frac{s}{2}} u^1_k\|^2_{L^2}  + \|(-\Delta)^{\frac{s}{2}} u^2_k\|^2_{L^2} +o_k(1) \\
	&\geq \frac{d}{d+2s} \left[ \left(\frac{a^*}{\mu} \right)^{\frac{2s}{d}} \|u^1_k\|^{\frac{4s}{d}+2}_{L^{\frac{4s}{d}+2}} +  \left( \frac{a^*}{a^*-\mu} \right)^{\frac{2s}{d}} \|u^2_k\|^{\frac{4s}{d}+2}_{L^{\frac{4s}{d}+2}}\right] +o_k(1) \\
	& \geq \frac{d}{d+2s} \min \left\{\left(\frac{a^*}{\mu} \right)^{\frac{2s}{d}}, \left( \frac{a^*}{a^*-\mu} \right)^{\frac{2s}{d}}  \right\} \left( \|u^1_k\|^{\frac{4s}{d}+2}_{L^{\frac{4s}{d}+2}} + \|u^2_k\|^{\frac{4s}{d}+2}_{L^{\frac{4s}{d}+2}} \right) + o_k(1) \\
	&= \frac{d}{d+2s} \min \left\{\left(\frac{a^*}{\mu} \right)^{\frac{2s}{d}}, \left( \frac{a^*}{a^*-\mu} \right)^{\frac{2s}{d}}  \right\}  \|u_{n_k}\|^{\frac{4s}{d}+2}_{L^{\frac{4s}{d}+2}} + o_k(1).
	\end{align*}
	Since $\mu \in (0,a^*)$, it follows from \eqref{assumption-compact-mini-fGN-2} that
	\[
	\liminf_{k\rightarrow \infty} \|(-\Delta)^{\frac{s}{2}} u_{n_k}\|^2_{L^2} - \frac{d}{d+2s} \|u_{n_k}\|^{\frac{4s}{d}+2}_{L^{\frac{4s}{d}+2}} >0 
	\]
	which contradicts to \eqref{assumption-compact-mini-fGN-1}. 
	
	\noindent {\bf Compactness.} Therefore, the compactness occurs. By Remark $\ref{rem-concen-compact}$, there exist a subsequence still denoted by $(u_{n_k})_{k\geq 1}$ and $(\tilde{y}_k)_{k\geq 1} \subset \R^d$ such that $(u_{n_k}(\cdot+\tilde{y}_k))_{k\geq 1}$ converges weakly in $H^s(\R^d)$ and strongly in $L^q$ for any $2\leq q<2^*$ to some function $u$. It follows that
	\[
	\|u\|^2_{L^2} = \lim_{k\rightarrow \infty} \|u_{n_k}(\cdot+\tilde{y}_k)\|^2_{L^2} = \lim_{k\rightarrow \infty} \|u_{n_k}\|^2_{L^2} = a^* 
	\]
	and by the lower semicontinuity of weak convergence,
	\begin{align*}
	\|(-\Delta)^{\frac{s}{2}} u\|_{L^2}^2 - \frac{d}{d+2s}\|u\|^{\frac{4s}{d}+2}_{L^{\frac{4s}{d}+2}} &\leq \liminf_{k\rightarrow \infty}  \|(-\Delta)^{\frac{s}{2}}[u_{n_k}(\cdot+\tilde{y}_k)]\|_{L^2}^2 - \frac{d}{d+2s} \|u_{n_k}(\cdot+\tilde{y}_k)\|^{\frac{4s}{d}+2}_{L^{\frac{4s}{d}+2}} \\
	&=\lim_{k\rightarrow \infty} \|(-\Delta)^{\frac{s}{2}}u_{n_k}\|_{L^2}^2 - \frac{d}{d+2s} \|u_{n_k}\|^{\frac{4s}{d}+2}_{L^{\frac{4s}{d}+2}} =0.
	\end{align*}
	By the fractional Gagliardo-Nirenberg inequality, we obtain
	\[
	\|(-\Delta)^{\frac{s}{2}} u\|^2_{L^2} - \frac{d}{d+2s}\|u\|^{\frac{4s}{d}+2}_{L^{\frac{4s}{d}+2}} =0
	\]
	which implies $\|(-\Delta)^{\frac{s}{2}} u\|_{L^2}=\lim_{k\rightarrow \infty} \|(-\Delta)^{\frac{s}{2}} [u_{n_k}(\cdot+\tilde{y}_k)]\|_{L^2}$. In particular, 
	\[
	u_{n_k}(\cdot+\tilde{y}_k) \rightarrow u \text{ strongly in } H^s(\R^d)
	\]
	as $k\rightarrow \infty$. We also have that $u$ is an optimizer for the fractional Gagliardo-Nirenberg inequality. By Theorem $\ref{theo-fractional-GN}$, $u(x) = \beta Q(\gamma(x+y))$ for some $\beta, \gamma>0$ and $y \in \R^d$. Since $\|u\|^2_{L^2} =a^*= \|Q\|^2_{L^2}$, we infer that $\beta=\gamma^{\frac{d}{2}}$, hence $u(x) = \gamma^{\frac{d}{2}}Q(\gamma(x+y))$ for some $\gamma>0$ and $y \in \R^d$. We thus obtain
	\[
	u_{n_k}(\cdot + y_k) \rightarrow\gamma^{\frac{d}{2}}Q(\gamma \cdot) \text{ strongly in } H^s(\R^d)
	\]
	as $k\rightarrow \infty$, where $y_k:= \tilde{y}_k - y$. The proof is complete.	
\end{proof}

\section{Existence and non-existence of minimizers} \label{section-existence}
\setcounter{equation}{0}

\subsection{No potential}
In this subsection, we show the existence and non-existence of minimizers for $I(a)$ in the case of no external potential.

\noindent {\bf Proof of Theorem $\ref{theo-no-potential}$.} 
{\bf The case $0<\alpha<\frac{4s}{d}$.} In this case, the existence of minimizers for $I(a)$ was established in \cite{Feng} via the concentration-compactness principle. We now give an alternative simple proof using the radial compact embedding. 

{\bf Step 1:} We first show that the minimizing problem $I(a)$ is well-defined and there exists $C>0$ such that $I(a)\leq -C<0$. To see this, we take $u \in H^s(\R^d)$ be such that $\|u\|^2_{L^2} =a$. By the fractional Gagliardo-Nirenberg inequality, we have
\[
\|u\|^{\alpha+2}_{L^{\alpha+2}} \lesssim \|(-\Delta)^{\frac{s}{2}} u\|^{\frac{d\alpha}{2s}}_{L^2} \|u\|^{\frac{4s-(d-2s)\alpha}{2}}_{L^2}
\]
which by Young's inequality and the fact $0<\frac{d\alpha}{2s}<2$ imply for any $\varep>0$,
\begin{align} \label{GN-Young-app}
\frac{1}{\alpha+2}\|u\|^{\alpha+2}_{L^{\alpha+2}} \leq \frac{\varep}{2} \|(-\Delta)^{\frac{s}{2}} u\|^2_{L^2} + C(\varep, \alpha, a).
\end{align}
This shows that for any $\varep>0$, there exists $C(\varep, \alpha,a)>0$ such that
\begin{align} \label{ener-est-no-potential-1}
E(u) \geq \frac{1}{2}(1-\varep)\|(-\Delta)^{\frac{s}{2}} u\|^2_{L^2} - C(\varep, \alpha,a).
\end{align}
By choosing $0<\varep<1$, we see that $E(u)\geq -C(\varep,\alpha,a)$. The minimizing problem $I(a)$ is thus well-defined. Next, we define 
\begin{align} \label{scaling}
u_\lambda(x)=\lambda^{\frac{d}{2}} u(\lambda x), \quad \lambda>0.
\end{align}
It follows that $\|u_\lambda\|^2_{L^2} = \|u\|^2_{L^2} =a$ and
\[
E(u_\lambda) = \frac{\lambda^{2s}}{2} \|(-\Delta)^{\frac{s}{2}}u\|^2_{L^2} - \frac{\lambda^{\frac{d\alpha}{2}}}{\alpha+2} \|u\|^{\alpha+2}_{L^{\alpha+2}}.
\]
Since $0<\frac{d\alpha}{2}<2s$, we can find $\lambda_0>0$ small enough so that $E(u_{\lambda_0}) <0$. Taking $C=-E(u_{\lambda_0})>0$, we get $I(a) \leq -C<0$.

{\bf Step 2:} We will show that there exists at least a minimizer for $I(a)$. Let $(u_n)_{n\geq 1}$ be a minimizing sequence for $I(a)$, i.e. $\|u_n\|^2_{L^2} =a$ for all $n\geq 1$ and $E(u_n) \rightarrow I(a)$ as $n\rightarrow \infty$. We may assume $u_n$ is radially symmetric and radially decreasing function. In fact, let $u^\ast_n$ be the symmetric rearrangement of $u_n$. Since the symmetric rearrangement preserves $L^q(\R^d)$ norm for any $1\leq q\leq \infty$ and by fractional Polya-Szeg\"o's inequality $\|(-\Delta)^{\frac{s}{2}} u^\ast_n\|_{L^2} \leq \|(-\Delta)^{\frac{s}{2}} u_n\|_{L^2}$, we see that $E(u_n^\ast) \leq E(u_n)$ and $\|u^\ast_n\|^2_{L^2}=\|u_n\|^2_{L^2} =a$. This shows that $(u^\ast_n)_{n\geq 1}$ is also a minimizing sequence for $I(a)$. Moreover, it follows from \eqref{ener-est-no-potential-1} that $(u_n)_{n\geq 1}$ is bounded in $H^s(\R^d)$. Indeed, since $E(u_n)\rightarrow I(a)$ as $n\rightarrow \infty$, there exists $C>0$ such that $E(u_n) \leq I(a) + C$ for  any $n\geq 1$. By \eqref{ener-est-no-potential-1}, we have that
\[
\frac{1}{2}(1-\varep) \|(-\Delta)^{\frac{s}{2}} u_n\|^2_{L^2} \leq E(u_n) + C(\varep, \alpha, a) \leq I(a) + C(\varep, \alpha, a)
\]
for any $n\geq 1$. This implies that $(u_n)_{n\geq 1}$ is bounded in $H^s(\R^d)$ by taking $0<\varep<1$. We thus obtain a bounded sequence in $H^s_{\text{rd}}(\R^d)$. By using the compact embedding $H^s_{\text{rd}}(\R^d) \hookrightarrow L^q(\R^d)$ for any $2<q<2^*$, there exist $u \in H^s(\R^d)$ and a subsequence $(u_{n_k})_{k\geq 1}$ such that $u_{n_k} \rightharpoonup u$ weakly in $H^s(\R^d)$ and $u_{n_k} \rightarrow u$ strongly in $L^q(\R^d)$ for any $q$ as above. We will show that $u$ is indeed a minimizer for $I(a)$. Since $I(a)<0$, we have that $u \ne 0$. In fact, assume by contradiction that $u\equiv 0$, then
\[
0\leq \liminf_{k\rightarrow \infty} \frac{1}{2} \|(-\Delta)^{\frac{s}{2}} u_{n_k}\|^2_{L^2}  = \liminf_{k\rightarrow \infty} E(u_{n_k}) = I(a) <0
\]
which is a contradiction. We have from the fact $u_{n_k} \rightharpoonup u$ weakly in $H^s(\R^d)$ and $u_{n_k} \rightarrow u$ strongly in $L^{\alpha+2}(\R^d)$ that
\[
E(u) \leq \liminf_{k\rightarrow \infty} E(u_{n_k}) = I(a).
\]
Moreover, by the lower semi-continuity of weak convergence,
\[
\|u\|^2_{L^2} \leq \liminf_{k\rightarrow \infty} \|u_{n_k}\|^2_{L^2} = a.
\]
Now set $\lambda = \sqrt{\frac{a}{\|u\|^2_{L^2}}} \geq 1$. We have that
\[
E(\lambda u) = \lambda^2 E(u) + \frac{\lambda^2(1-\lambda^\alpha)}{\alpha+2} \|u\|^{\alpha+2}_{L^{\alpha+2}}
\]
or
\[
E(u) = \frac{E(\lambda u)}{\lambda^2} + \frac{\lambda^\alpha-1}{\alpha+2} \|u\|^{\alpha+2}_{L^{\alpha+2}}.
\]
Since $u\ne 0$ and $\|\lambda u\|^2_{L^2} =a$, it follows that
\[
I(a) \geq E(u) \geq \frac{E(\lambda u)}{\lambda^2} \geq \frac{I(a)}{\lambda^2}.
\]
This implies that $\lambda \leq 1$ since $I(a)<0$, hence $\lambda =1$ and $\|u\|^2_{L^2} =a$. We thus get $I(a) \leq E(u)$, hence $E(u) = I(a)$ or $u$ is a minimizer for $I(a)$. This shows that for any $a>0$, there exists at least a minimizer for $I(a)$ and $-\infty<I(a)<0$.

{\bf The case $\alpha=\frac{4s}{d}$.} We first show the non-existence of minimizers for $I(a)$ in the case $0<a<a^*$. Assume that there exists a minimizer $u$ for $I(a)$ with $0<a<a^*$. By the Gagliardo-Nirenberg inequality,
\begin{align} \label{non-min-Ia-1}
I(a) = E(u) \geq \frac{1}{2} \left(1 - \left(\frac{a}{a^*}\right)^{\frac{2s}{d}} \right) \|(-\Delta)^{\frac{s}{2}} u\|^2_{L^2} >0.
\end{align}
On the other hand, we take $u \in H^s(\R^d)$ satisfying $\|u\|^2_{L^2}=a$ and consider $u_\lambda$ as in \eqref{scaling}. We see that $\|u_\lambda\|^2_{L^2}=\|u\|^2_{L^2}=a$ and
\[
I(a) \leq E(u_\lambda) = \frac{\lambda^{2s}}{2} \|(-\Delta)^{\frac{s}{2}} u\|^2_{L^2} - \frac{\lambda^{2s}}{2(d+2s)} \|u\|^{\frac{4s}{d}+2}_{L^{\frac{4s}{d}+2}} \rightarrow 0
\]
as $\lambda \rightarrow 0$ which contradicts \eqref{non-min-Ia-1}.

We next prove the non-existence of minimizers for $I(a)$ when $a>a^*$. To this end, we need the following estimate due to \cite[Lemma 3.2]{DTWZ}.
\begin{lemma}[\cite{DTWZ}] \label{lem-decay-est}
	Let $d\geq 1$ and $s\in (0,1)$. Let $\varphi: \R^d \rightarrow \R$ be a smooth compactly supported function satisfying $0\leq \varphi \leq 1$, $\varphi=1$ on $|x| \leq 1$. Denote for $\tau>0$,
	\begin{align} \label{defi-Q-tau}
	Q_\tau(x):= \varphi(\tau^{-1} x) Q(x),
	\end{align}
	where $Q$ is the unique (up to translations) positive radial ground state related to \eqref{ell-equ-fra-cri}. Then it holds that
	\[
	\|(-\Delta)^{\frac{s}{2}} Q_\tau \|^2_{L^2} = \|(-\Delta)^{\frac{s}{2}} Q\|^2_{L^2} + O(\tau^{-4s})
	\]
	as $\tau \rightarrow \infty$.
\end{lemma}
This estimate is in fact a refined estimate of \eqref{product-est} thanks to the exact decay of $Q$ at infinity given in Theorem $\ref{theo-fractional-GN}$. Note that in \cite[Lemma 3.2]{DTWZ}, the above estimate is stated in dimensions $d\geq 2$, but it still holds in 1D. 

Let $\varphi$ be as in Lemma $\ref{lem-decay-est}$ and denote
\begin{align} \label{test-function}
u_\tau(x) = A_\tau \tau^{\frac{d}{2}} \varphi(x) Q_0(\tau x), \quad \tau>0,
\end{align}
where $Q_0 = \frac{Q}{\|Q\|_{L^2}}$ and $A_\tau>0$ is such that $\|u_\tau\|^2_{L^2} =a$ for all $\tau>0$. By definition,
\[
a A^{-2}_\tau = \int \varphi^2(\tau^{-1}x) Q_0^2(x) dx.
\]
Since $\|Q_0\|_{L^2} =1$ and $Q_0(x), |\nabla Q_0(x)| = O\left( |x|^{-d-2s} \right)$ for some $c>0$ as $|x|\rightarrow \infty$, we see that for $\tau$ sufficiently large,
\[
\left|\int (1-\varphi^2(\tau^{-1}x)) Q_0^2(x) dx \right| \lesssim \int_{|x|\geq \tau} |x|^{-2d-4s} dx\lesssim \int_{|x|\geq \tau} |x|^{-2d-4s} dx \lesssim \tau^{-d-4s}.
\]
This shows that
\begin{align} \label{est-A-tau}
a A^{-2}_\tau = 1 + O\left(\tau^{-d-4s}\right) \quad \text{or} \quad A^2_\tau = a + O\left(\tau^{-d-4s}\right)
\end{align}
as $\tau \rightarrow \infty$, where $A=O(\tau^{-d-4s})$ means $A\leq C \tau^{-d-4s}$ for some constant $C>0$ independent of $\tau$. Using the fact
\[
(-\Delta)^{\frac{s}{2}} u_\tau (x) = A_\tau  \frac{\tau^{\frac{d}{2}}}{\|Q\|_{L^2}} (-\Delta)^{\frac{s}{2}} [\varphi(x) Q(\tau x)] = A_\tau \tau^s \frac{\tau^{\frac{d}{2}}}{\|Q\|_{L^2}}  [(-\Delta)^{\frac{s}{2}} Q_\tau] (\tau x),
\]
Lemma $\ref{lem-decay-est}$ implies that
\begin{align*}
\|(-\Delta)^{\frac{s}{2}} u_\tau\|^2_{L^2} &= \frac{A_\tau^2 \tau^{2s}}{\|Q\|^2_{L^2}} \|(-\Delta)^{\frac{s}{2}} Q_\tau\|^2_{L^2} \\
&= \frac{A_\tau^2 \tau^{2s}}{\|Q\|^2_{L^2}} \left( \|(-\Delta)^{\frac{s}{2}} Q\|^2_{L^2} + O(\tau^{-4s})\right)
\end{align*}
as $\tau \rightarrow \infty$. By \eqref{est-A-tau}, we see that
\begin{align} \label{sobo-norm-test-function}
\|(-\Delta)^{\frac{s}{2}} u_\tau\|^2_{L^2} = \tau^{2s} a \|(-\Delta)^{\frac{s}{2}} Q_0\|^2_{L^2} + O(\tau^{-d-6s})
\end{align}
as $\tau \rightarrow \infty$. Similarly,
\begin{align} \label{nonli-norm-test-function}
\|u_\tau\|^{\frac{4s}{d}+2}_{L^{\frac{4s}{d}+2}} = \tau^{2s} a^{\frac{2s}{d}+1}_\tau \|Q_0\|^{\frac{4s}{d}+2}_{L^{\frac{4s}{d}+2}} + O(\tau^{-d-6s})
\end{align}
as $\tau \rightarrow \infty$. By the definition of $I(a)$,
\begin{align}
\frac{I(a)}{a} \leq \frac{E(u_\tau)}{a} &= \frac{\tau^{2s}}{2} \left( \|(-\Delta)^{\frac{s}{2}} Q_0\|^2_{L^2} - a^{\frac{2s}{d}}\frac{ d}{d+2s} \|Q_0\|^{\frac{4s}{d}+2}_{L^{\frac{4s}{d}+2}} \right) + O(\tau^{-d-6s}) \nonumber\\
&=\tau^{2s}\frac{d}{4s}\left(1-\left(\frac{a}{a^*}\right)^{\frac{2s}{d}} \right) + O(\tau^{-d-6s}) \label{est-Ia-no-potential}
\end{align}
as $\tau \rightarrow \infty$. Here we have used \eqref{pohozaev-identity} to get 
\[
\frac{d}{2s}= \|(-\Delta)^{\frac{s}{2}} Q_0\|^2_{L^2} = \frac{d}{d+2s} \|Q_0\|^{\frac{4s}{d}+2}_{L^{\frac{4s}{d}+2}} \|Q\|^{\frac{4s}{d}}_{L^2}.
\]
It follows from \eqref{est-Ia-no-potential} that for $a>a^*$,
\[
I(a) \leq \lim_{\tau \rightarrow \infty} E(u_\tau) = -\infty.
\]
This shows that $I(a) =-\infty$ for $a>a^*$, hence there is no minimizer for $I(a)$ when $a>a^*$. 

We now consider the case $a=a^*$, we also have from \eqref{est-Ia-no-potential} that $I(a^*) \leq 0$ which together with $I(a^*)\geq 0$ (by the fractional Gagliardo-Nirenberg inequality) imply $I(a)=0$. In this case, we observe that $Q$ is a minimizer for $I(a^*)$ since $E(Q)=0$. If $u$ is another minimizer for $I(a^*)$, then $u$ is an optimizer for the fractional Gagliardo-Nirenberg inequality, hence by Theorem $\ref{theo-fractional-GN}$, $u(x) = \beta Q(\gamma(x+y))$ for some $\beta \in \C, \beta \ne 0, \gamma>0$ and $y \in \R^d$. Since $\|u\|_{L^2}^2= a^*=\|Q\|^2_{L^2}$, we infer that $|\beta|= \gamma^{\frac{d}{2}}$ or $u(x) = e^{i\theta} \gamma^{\frac{d}{2}} Q(\gamma x+x^0)$ for some $\theta \in \R, \gamma>0$ and $x^0 \in \R^d$.

{\bf The case $\frac{4s}{d}<\alpha<s^*$.} Let $u \in H^s(\R^d)$ be such that $\|u\|^2_{L^2} =a$. Let $u_\lambda$ be as in \eqref{scaling}. We see that $\|u_\lambda\|^2_{L^2}= \|u\|^2_{L^2}=a$ and
\[
E(u_\lambda) = \frac{\lambda^{2s}}{2} \|(-\Delta)^{\frac{s}{2}} u\|^2_{L^2} -\frac{\lambda^{\frac{d\alpha}{2}}}{\alpha+2} \|u\|^{\alpha+2}_{L^{\alpha+2}}.
\]
Since $\frac{d\alpha}{2}>2s$, we see that $E(u_\lambda) \rightarrow -\infty$ as $\lambda \rightarrow \infty$. This implies that $I(a) = -\infty$ and thus there is no minimizer for $I(a)$. 
\hfill $\Box$

\subsection{Periodic potential}
In this subsection, we show the existence and non-existence of minimizers for $I(a)$ in the case of periodic potentials.

\noindent {\bf Proof of Theorem $\ref{theo-existence}$.} 
{\bf The case $0<\alpha<\frac{4s}{d}$.} Let $u \in H^s(\R^d)$ be such that $\|u\|^2_{L^2} =a$. By the fractional Gagliardo-Nirenberg inequality and Young's inequality with $0<\frac{d\alpha}{2s} <2$ (see \eqref{GN-Young-app}), we see that
\begin{align} \label{ener-est-periodic-potential}
E(u) \geq \frac{1}{2}(1-\varep) \|(-\Delta)^{\frac{s}{2}} u\|^2_{L^2} + \frac{a}{2} \min V - C(\varep, \alpha, a).
\end{align}
Taking $0<\varep<1$, we see that $E(u)\geq \frac{a}{2} \min V  - C(\varep, \alpha,a)$ for all $u \in H^s(\R^d)$ with $\|u\|^2_{L^2} =a$. This shows that $I(a) >-\infty$. 

Let $x^0 \in \R^d$ and $\varphi$ be as in Lemma $\ref{lem-decay-est}$. Denote
\begin{align} \label{test-function-x0}
u_\tau(x) = A_\tau \tau^{\frac{d}{2}} \varphi(x-x^0) Q^\alpha_0(\tau(x-x^0)), \quad \tau>0,
\end{align}
where $Q^\alpha_0 = \frac{Q_\alpha}{\|Q_\alpha\|_{L^2}}$ and $A_\tau>0$ is such that $\|u_\tau\|^2_{L^2} =a$ for all $\tau>0$. By the same argument as in \eqref{est-A-tau}, we have that
\[
A_\tau^2 = a + O(\tau^{-d-4s})
\]
as $\tau \rightarrow \infty$. We also have from \eqref{sobo-norm-test-function} and \eqref{nonli-norm-test-function} that
\[
\|(-\Delta)^{\frac{s}{2}} u_\tau\|^2_{L^2} = \tau^{2s} a \|(-\Delta)^{\frac{s}{2}} Q^\alpha_0\|^2_{L^2} + O(\tau^{-d-6s})
\]
and
\[
\|u_\tau\|^{\alpha+2}_{L^{\alpha+2}} = \tau^{\frac{d\alpha}{2}} a^{\frac{\alpha}{2} +1} \|Q^\alpha_0\|^{\alpha+2}_{L^{\alpha+2}} + O(\tau^{-d-6s})
\]
as $\tau \rightarrow \infty$. On the other hand, we note that $\tau^d [Q^\alpha_0(\tau(x-x^0))]^2$ converges weakly to the Dirac delta function at $x^0$ when $\tau \rightarrow \infty$. Indeed, we take $\psi \in C^\infty_0(\R^d)$ and compute
\begin{align*}
\int \tau^d [Q^\alpha_0(\tau(x-x^0))]^2 \psi(x) dx &= \int \tau^d [Q^\alpha_0(\tau(x-x^0))]^2(\psi(x) - \psi(x^0)) dx + \int \tau^d [Q^\alpha_0(\tau(x-x^0))]^2 \psi(x^0) dx \\
&= \int [Q^\alpha_0(x)]^2 \left(\psi(\tau^{-1} x + x^0) - \psi(x^0) \right) dx + \psi(x^0)
\end{align*}
where we have used that $\|Q^\alpha_0\|^2_{L^2}=1$. Note that the integral tends to zero as $\tau \rightarrow \infty$ due to the dominated convergence. Since $x\mapsto V(x) [\varphi(x-x^0)]^2$ is integrable, it follows that
\[
\int V|u_\tau|^2 dx = A_\tau^2 \int V(x) [\varphi(x-x^0)]^2 \tau^d [Q^\alpha_0(\tau(x-x^0))]^2 dx \rightarrow a V(x^0)
\]
as $\tau \rightarrow \infty$ for almost everywhere $x^0 \in \R^d$. By the definition of $I(a)$, we have that
\begin{align}
\frac{I(a)}{a} \leq \frac{E(u_\tau)}{a} &= \frac{\tau^{2s}}{2} \|(-\Delta)^{\frac{s}{2}} Q^\alpha_0\|^2_{L^2} - \frac{\tau^{\frac{d\alpha}{2}}}{\alpha+2} a^{\frac{\alpha}{2}} \|Q^\alpha_0\|^{\alpha+2}_{L^{\alpha+2}} + \frac{1}{2} V(x^0) + o_\tau(1) \nonumber \\
&= \frac{d\alpha}{4s-(d-2s)\alpha} \left(\frac{\tau^{2s}}{2} - \frac{2s \tau^{\frac{d\alpha}{2}}}{d\alpha} \left(\frac{a}{\|Q_\alpha\|^2_{L^2}}\right)^{\frac{\alpha}{2}} \right) + \frac{1}{2} V(x^0) + o_\tau(1) \label{est-Ia-periodic-potential-sub}
\end{align}
as $\tau \rightarrow \infty$ for almost everywhere $x^0 \in \R^d$. Here we have used \eqref{pohozaev-identity-alpha} to get
\[
\frac{d\alpha}{4s-(d-2s)\alpha} = \|(-\Delta)^{\frac{s}{2}} Q^\alpha_0\|^2_{L^2} = \frac{d\alpha}{2s(\alpha+2)} \|Q_\alpha\|^\alpha_{L^2} \|Q^\alpha_0\|^{\alpha+2}_{L^{\alpha+2}}.
\]
Taking $\tau=\left(\frac{a}{\|Q_\alpha\|^2_{L^2}}\right)^{\frac{\alpha}{4s-d\alpha}}$ and noting that $\tau \rightarrow \infty$ as $a \rightarrow \infty$, we infer from \eqref{est-Ia-periodic-potential-sub} that
\begin{align} \label{est-Ia-periodic-potential-sub-1}
\frac{I(a)}{a} \leq - \frac{4s-d\alpha}{4s-(d-2s)\alpha} \left(\frac{a}{\|Q_\alpha\|^2_{L^2}} \right)^{\frac{2s\alpha}{4s-d\alpha}} + \frac{1}{2} V(x^0) + o_a(1)
\end{align}
as $a \rightarrow \infty$ for almost everywhere $x^0 \in \R^d$. Since $\frac{I(a)}{a} \rightarrow -\infty$ as $a \rightarrow \infty$, there exists $a_*>0$ large enough such that
\begin{align} \label{est-Ia-periodic-potential-sub-3}
\frac{I(a)}{a} < \frac{1}{2} \inf \sigma((-\Delta)^s +V)
\end{align}
for any $a>a_*$. We will show the existence of minimizers for $I(a)$ with $a>a_*$. Let $(u_n)_{n\geq 1}$ be a minimizing sequence for $I(a)$. By \eqref{ener-est-periodic-potential}, $(u_n)_{n\geq 1}$ is a bounded sequence in $H^s(\R^d)$. By the concentration-compactness principle given in Lemma $\ref{lem-concen-compact}$ (see also Remark $\ref{rem-concen-compact-all-k}$), there exists a subsequence $(u_{n_k})_{k\geq 1}$ satisfying one of the following three possibilities: vanishing, compactness and dichotomy. 

{\bf No vanishing.} If $(u_{n_k})_{k\geq 1}$ is vanishing, then by Remark $\ref{rem-concen-compact}$, $u_{n_k} \rightarrow 0$ strongly in $L^q(\R^d)$ for any $2<q<2^*$.  Thus 
\[
I(a) = \lim_{k\rightarrow \infty} E(u_{n_k}) = \liminf_{k\rightarrow \infty} \left(\frac{1}{2} \|(-\Delta)^{\frac{s}{2}} u_{n_k}\|^2_{L^2} + \frac{1}{2} \int V |u_{n_k}|^2 dx \right) \geq \frac{a}{2} \inf \sigma((-\Delta)^s +V)
\]
which contradicts to \eqref{est-Ia-periodic-potential-sub-3}. 

{\bf No dichotomy.} Assume the dichotomy occurs. Let $(u^1_k)_{k\geq 1}$ and $(u^2_k)_{k\geq 1}$ be the corresponding sequences in Lemma $\ref{lem-concen-compact}$. We first claim that
\begin{align} \label{ener-est-periodic-potential-1}
\liminf_{k\rightarrow \infty} E(u_{n_k}) - E(u^1_k) - E(u^2_k) \geq 0.
\end{align}
By Lemma $\ref{lem-concen-compact}$, we have that
\[
\liminf_{k\rightarrow \infty} \|(-\Delta)^{\frac{s}{2}} u_{n_k}\|^2_{L^2} -\|(-\Delta)^{\frac{s}{2}} u^1_k\|^2_{L^2}-\|(-\Delta)^{\frac{s}{2}} u^2_k\|^2_{L^2} \geq 0.
\]
Since $u^1_k$ and $u^2_k$ have disjoint supports for $k$ large and $\lim_{k\rightarrow \infty} \|u_{n_k} - u^1_k - u^2_k\|_{L^q} =0$ for $2 \leq q<2^*$, we see that
\[
\|u_{n_k}\|^{\alpha+2}_{L^{\alpha+2}} - \|u^1_k\|^{\alpha+2}_{L^{\alpha+2}} - \|u^2_k\|^{\alpha+2}_{L^{\alpha+2}} \rightarrow 0
\]
as $k\rightarrow \infty$. Since $V$ is bounded, 
\[
\int V(|u_{n_k}|^2 - |u^1_k|^2 - |u^2_k|^2) dx = \int V(|u_{n_k}|^2 - |u^1_k + u^2_k|^2) dx \rightarrow 0
\]
as $k\rightarrow \infty$. Collecting the above estimates, we prove \eqref{ener-est-periodic-potential-1}.

We next have for any $\lambda>0$ that
\begin{align*}
E(\lambda u) &= \lambda^2 \left( \frac{1}{2} \|(-\Delta)^{\frac{s}{2}} u\|^2_{L^2} + \frac{1}{2} \int V|u|^2 dx \right) - \frac{\lambda^{\alpha+2}}{\alpha+2} \|u\|^{\alpha+2}_{L^{\alpha+2}} \\
&=\lambda^{\alpha+2} E(u) + \lambda^2(1-\lambda^\alpha) \left(\frac{1}{2} \|(-\Delta)^{\frac{s}{2}} u\|^2_{L^2} + \frac{1}{2} \int V|u|^2 dx \right)
\end{align*}
or equivalently
\[
E(u) = \frac{1}{\lambda^{\alpha+2}} E(\lambda u) + \left(1-\frac{1}{\lambda^\alpha}\right) \left(\frac{1}{2} \|(-\Delta)^{\frac{s}{2}} u\|^2_{L^2} + \frac{1}{2} \int V|u|^2 dx \right).
\]
Set
\[
\lambda^1_k := \sqrt{\frac{a}{\|u^1_k\|^2_{L^2}}}, \quad \lambda^2_k:= \sqrt{\frac{a}{\|u^2_k\|^2_{L^2}}}.
\]
It follows that
\begin{align*}
E(u^1_k) &= \frac{1}{(\lambda^1_k)^{\alpha+2}} E(\lambda^1_k u^1_k) + \left(1-\frac{1}{(\lambda^1_k)^\alpha}\right) \left( \frac{1}{2} \|(-\Delta)^{\frac{s}{2}} u^1_k\|^2_{L^2} + \frac{1}{2} \int V|u^1_k|^2 dx\right) \\
&\geq \left(\frac{\|u^1_k\|^2_{L^2}}{a}\right)^{\frac{\alpha}{2}+1} I(a) + \|u^1_k\|^2_{L^2} \left(1- \left(\frac{\|u^1_k\|^2_{L^2}}{a}\right)^{\frac{\alpha}{2}} \right) \frac{1}{2} \inf \sigma((-\Delta)^s + V).
\end{align*}
Using the fact that $\|u^1_k\|^2_{L^2} \rightarrow \mu$ as $k \rightarrow \infty$, we see that
\[
E(u^1_k) \geq \left(\frac{\mu}{a}\right)^{\frac{\alpha}{2}+1} I(a) + \frac{\mu}{a}\left(1- \left(\frac{\mu}{a}\right)^{\frac{\alpha}{2}} \right) \frac{a}{2} \inf \sigma((-\Delta)^s + V) + o_k(1)
\]
as $k \rightarrow \infty$. Similar, since $\|u^2_k\|^2_{L^2} \rightarrow a-\mu$ as $k\rightarrow \infty$, 
\[
E(u^2_k) \geq \left(\frac{a-\mu}{a}\right)^{\frac{\alpha}{2}+1} I(a) + \frac{a-\mu}{a}\left(1- \left(\frac{a-\mu}{a}\right)^{\frac{\alpha}{2}} \right) \frac{a}{2} \inf \sigma((-\Delta)^s + V) + o_k(1)
\]
as $k \rightarrow \infty$. We infer from \eqref{ener-est-periodic-potential-1} that
\begin{align*}
E(u_{n_k}) &\geq E(u^1_k) + E(u^2_k) + o_k(1) \\
& \geq \left( \left(\frac{\mu}{a}\right)^{\frac{\alpha}{2} +1} + \left(\frac{a-\mu}{a}\right)^{\frac{\alpha}{2}+1} \right) I(a) \\
& \mathrel{\phantom{\geq}} + \left( 1- \left(\frac{\mu}{a}\right)^{\frac{\alpha}{2} +1} - \left(\frac{a-\mu}{a}\right)^{\frac{\alpha}{2}+1} \right) \frac{a}{2} \inf \sigma((-\Delta)^s+V) + o_k(1)
\end{align*}
as $k \rightarrow \infty$. Taking $k\rightarrow \infty$, we obtain
\[
I(a) \geq \left( \left(\frac{\mu}{a}\right)^{\frac{\alpha}{2} +1} + \left(\frac{a-\mu}{a}\right)^{\frac{\alpha}{2}+1} \right) I(a) + \left( 1- \left(\frac{\mu}{a}\right)^{\frac{\alpha}{2} +1} - \left(\frac{a-\mu}{a}\right)^{\frac{\alpha}{2}+1} \right) \frac{a}{2} \inf \sigma((-\Delta)^s+V).
\]
Since $0<\mu<a$, we obtain
\[
I(a) \geq \frac{a}{2} \inf \sigma((-\Delta)^s+V)
\]
which contradicts to \eqref{est-Ia-periodic-potential-sub-3}.

{\bf Compactness.} There thus exists a sequence $(y_k)_{k\geq 1} \subset \R^d$ such that up to a subsequence, $(u_{n_k}(\cdot + y_k))_{k\geq 1}$ converges weakly in $H^s(\R^d)$ and strongly in $L^q(\R^d)$ to some $u$ for all $2\leq q<2^*$. This implies that
\[
\|u\|^2_{L^2} = \lim_{k\rightarrow \infty} \|u_{n_k}(\cdot+y_k)\|^2_{L^2} =a
\]
and
\[
\|(-\Delta)^{\frac{s}{2}} u \|^2_{L^2} \leq \liminf_{k\rightarrow \infty} \|(-\Delta)^{\frac{s}{2}} u_{n_k}(\cdot+y_k)\|^2_{L^2} = \liminf_{k\rightarrow \infty} \|(-\Delta)^{\frac{s}{2}} u_{n_k}\|^2_{L^2}
\]
and
\[
\|u\|^{\alpha+2}_{L^{\alpha+2}} = \lim_{k\rightarrow \infty} \|u_{n_k}(\cdot+y_k)\|^{\alpha+2}_{L^{\alpha+2}} = \lim_{k\rightarrow \infty} \|u_{n_k}\|^{\alpha+2}_{L^{\alpha+2}}.
\]
We next have that
\begin{align*}
\int V(x)|u_{n_k}(x)|^2 dx &= \int V(x+y_k) |u_{n_k}(x+y_k)|^2 dx \\
&= \int V(x+y_k) |u(x)|^2 dx + \int V(x+y_k) \left( |u_{n_k}(x+y_k)|^2 - |u(x)|^2\right) dx.
\end{align*}
Note that
\begin{multline*}
\left| \int V(x+y_k) \left( |u_{n_k}(x+y_k)|^2 - |u(x)|^2 \right)  \right| \\
\leq \|V\|_{L^\infty} \|u_{n_k}(\cdot+y_k) - u\|_{L^2} \left(\|u_{n_k}(\cdot+y_k)\|_{L^2} + \|u\|_{L^2} \right) \rightarrow 0
\end{multline*}
as $k\rightarrow \infty$. This implies that
\begin{align*}
E(u_{n_k}) &= \frac{1}{2} \|(-\Delta)^{\frac{s}{2}} u_{n_k}\|^2_{L^2} +\frac{1}{2} \int V|u_{n_k}|^2 dx - \frac{1}{\alpha+2} \|u_{n_k}\|^{\alpha+2}_{L^{\alpha+2}} \\
&\geq \frac{1}{2} \|(-\Delta)^{\frac{s}{2}} u\|^2_{L^2} + \frac{1}{2} \int V(x+y_k) |u(x)|^2 dx - \frac{1}{\alpha+2} \|u\|^{\alpha+2}_{L^{\alpha+2}} + o_k(1)
\end{align*}
as $k \rightarrow \infty$. Since $V$ is periodic, we write
\[
y_k = x_k + z_k, \quad x_k \in [0,1]^d, z_k \in \Z^d.
\]
Since $(x_k)_{k\geq 1}$ is bounded in $\R^d$, up to a subsequence, $x_k \rightarrow x^0$ as $k\rightarrow \infty$. Thus
\[
\int V(x+y_k) |u(x)|^2 dx = \int V(x+x_k) |u(x)|^2 dx \rightarrow \int V(x+x^0) |u(x)|^2 dx
\]
as $k\rightarrow \infty$ by the dominated convergence. Hence
\begin{align*}
E(u_{n_k}) &\geq \frac{1}{2} \|(-\Delta)^{\frac{s}{2}} u\|^2_{L^2} + \frac{1}{2} \int V(x+x^0) |u(x)|^2 dx - \frac{1}{\alpha+2} \|u\|^{\alpha+2}_{L^{\alpha+2}} + o_k(1) \\
&= E(u(\cdot-x^0)) + o_k(1)
\end{align*}
as $k \rightarrow \infty$. Taking $k\rightarrow \infty$, we get $I(a) \geq E(u(\cdot-x^0))$. On the other hand, $\|u(\cdot-x^0)\|^2_{L^2} = \|u\|^2_{L^2} =a$ which implies that $I(a) \leq E(u(\cdot-x^0))$. This shows that $u(\cdot-x^0)$ is a minimizer for $I(a)$. This shows the existence of minimizers for $I(a)$ for $a>a_*$.

{\bf The case $\alpha=\frac{4s}{d}$.} We first claim that
\begin{align} \label{est-Ia-periodic-potential-cri-1}
\lim_{a \nearrow a^*} \frac{I(a)}{a} = \frac{I(a^*)}{a^*}= \frac{1}{2} \min V.
\end{align}
Let $u_\tau$ be as in \eqref{test-function-x0}. By the same argument as in the mass-subcritical case (see also \eqref{est-Ia-no-potential}), we have that
\begin{align*}
\|(-\Delta)^{\frac{s}{2}} u_\tau\|^2_{L^2} &= \tau^{2s} a \|(-\Delta)^{\frac{s}{2}} Q_0\|^2_{L^2} + O(\tau^{-d-6s}), \\
\|u_\tau\|^{\frac{4s}{d}+2}_{L^{\frac{4s}{d}+2}} &= \tau^{2s} a^{\frac{2s}{d}+1} \|Q_0\|^{\frac{4s}{d}+2}_{L^{\frac{4s}{d}+2}} + O(\tau^{-d-6s}),
\end{align*}
where $Q_0= \frac{Q}{\|Q\|_{L^2}}$ and
\[
\int V|u_\tau|^2 dx \rightarrow a V(x^0)
\]
as $\tau \rightarrow \infty$ for almost everywhere $x^0 \in \R^d$. It follows that
\begin{align}
\frac{I(a)}{a} \leq \frac{E(u_\tau)}{a} &= \frac{\tau^{2s}}{2} \left(\|(-\Delta)^{\frac{s}{2}} Q_0\|^2_{L^2} - a^{\frac{2s}{d}} \frac{d}{d+2s} \|Q_0\|^{\frac{4s}{d}+2}_{L^{\frac{4s}{d}+2}} \right) + \frac{1}{2} V(x^0) + o_\tau(1) \nonumber \\
&=\tau^{2s} \frac{d}{4s} \left( 1- \left(\frac{a}{a^*}\right)^{\frac{2s}{d}} \right) + \frac{1}{2} V(x^0) + o_\tau(1) \label{est-Ia-periodic-potential-cri-2}
\end{align}
as $\tau \rightarrow \infty$ for almost everywhere $x^0 \in \R^d$. Taking $\tau= \left(1-\left(\frac{a}{a^*}\right)^{\frac{2s}{d}} \right)^{-\frac{1}{4s}}$ so that $\tau \rightarrow \infty$ as $a \nearrow a^*$, we get
\begin{align} \label{est-Ia-periodic-cri}
\frac{I(a)}{a} \leq \frac{d}{4s} \left(1-\left(\frac{a}{a^*}\right)^{\frac{2s}{d}} \right)^{\frac{1}{2}} + \frac{1}{2} V(x^0) + o_{a \nearrow a^*}(1)
\end{align}
as $a \nearrow a^*$ for almost everywhere $x^0 \in \R^d$. Letting $a\nearrow a^*$ and optimizing the right hand side, we obtain
\[
\limsup_{a \nearrow a^*} \frac{I(a)}{a} \leq \frac{1}{2} \min V.
\]
On the other hand, by the fractional Gagliardo-Nirenberg inequality, we have for $\|u\|^2_{L^2}=a$ that
\begin{align*} 
E(u) \geq \frac{1}{2} \left(1- \left(\frac{a}{a^*}\right)^{\frac{2s}{d}} \right) \|(-\Delta)^{\frac{s}{2}} u\|^2_{L^2} + \frac{a}{2} \min V.
\end{align*}
Assume at the moment that $\frac{I(a)}{a}$ is a decreasing function in $a$. This implies that
\[
\frac{I(a)}{a} \geq \frac{I(a^*)}{a^*} \geq \frac{1}{2} \min V,
\]
hence
\[
\liminf_{a \nearrow a^*} \frac{I(a)}{a} \geq \frac{I(a^*)}{a^*} \geq \frac{1}{2} \min V
\]
which proves the claim. To see that $a \mapsto \frac{I(a)}{a}$ is a decreasing function, we take $0<a \leq b$. We will show that $\frac{I(a)}{a} \geq \frac{I(b)}{b}$. Let $u \in H^s(\R^d)$ be such that $\|u\|^2_{L^2} =a$ and set $\lambda = \sqrt{\frac{b}{a}} \geq 1$. We see that $\|\lambda u\|^2_{L^2}=b$ and 
\begin{align*}
E(\lambda u) &= \lambda^2 \left(\frac{1}{2} \|(-\Delta)^{\frac{s}{2}} u\|^2_{L^2} + \frac{1}{2} \int V|u|^2 dx \right) - \frac{\lambda^{\alpha+2}}{\alpha+2} \|u\|^{\alpha+2}_{L^{\alpha+2}} \\
&= \lambda^2 E(u) + \frac{\lambda^2(1-\lambda^\alpha)}{\alpha+2} \|u\|^{\alpha+2}_{L^{\alpha+2}} \leq \lambda^2 E(u).
\end{align*}
By the definition of $I(b)$, 
\[
I(b) \leq E(\lambda u) \leq \lambda^2 E(u).
\]
Taking the infimum over all $u \in H^s(\R^d)$ with $\|u\|^2_{L^2} =a$, we get $I(b) \leq \frac{b}{a} I(a)$ which shows that $\frac{I(a)}{a}$ is a decreasing function in $a$. 

It also follows from \eqref{est-Ia-periodic-potential-cri-2} that for $a>a^*$,
\[
\frac{I(a)}{a} \leq \lim_{\tau \rightarrow \infty} \frac{E(u_\tau)}{a} = -\infty
\]
which implies that there is no minimizer for $I(a)$.

We next show that there is no minimizer for $I(a^*)$. In fact, assume by contradiction that there exists a minimizer for $I(a^*)$, says $u$. Then by the fractional Gagliardo-Nirenberg inequality,
\begin{align*}
\frac{a^*}{2} \min V = I(a^*) &=E(u) = \frac{1}{2} \|(-\Delta)^{\frac{s}{2}} u\|^2_{L^2} +\frac{1}{2} \int V|u|^2 dx - \frac{d}{2(d+2s)} \|u\|^{\frac{4s}{d}+2}_{L^{\frac{4s}{d}+2}} \geq \frac{a^*}{2} \min V.
\end{align*}
This implies that
\begin{align} \label{no-minimizer-periodic-potential-cri-1}
\int V|u|^2 dx = a^* \min V
\end{align}
and
\begin{align} \label{no-minimizer-periodic-potential-cri-2}
\frac{1}{2} \|(-\Delta)^{\frac{s}{2}} u\|^2_{L^2} - \frac{d}{2(d+2s)} \|u\|^{\frac{4s}{d}+2}_{L^{\frac{4s}{d}+2}} =0.
\end{align}
By \eqref{no-minimizer-periodic-potential-cri-2}, $u$ is an optimizer for the fractional Gagliardo-Nirenberg inequality, hence $u$ is equal to $Q$ up to symmetries. In this case, \eqref{no-minimizer-periodic-potential-cri-1} cannot occur except $V$ is a constant, but it contradicts to the assumption $\min V <\inf \sigma((-\Delta)^s+V)$.

We next show the existence of minimizers for $I(a)$ when $a_*<a<a^*$ for some $0<a_*<a^*$. By \eqref{est-Ia-periodic-potential-cri-1} and the assumption $\min V<\inf \sigma((-\Delta)^s +V)$, there exists $0<a_*<a^*$ such that
\begin{align} \label{est-Ia-periodic-potential-cri-3}
\frac{I(a)}{a} <\frac{1}{2} \inf \sigma((-\Delta)^s+V)
\end{align}
for any $a_*<a<a^*$. Using \eqref{est-Ia-periodic-potential-cri-3}, we can repeat the same argument as in the mass-subcritical case to show the existence of minimizers for $I(a)$ when $a_*<a<a^*$.

{\bf The case $\frac{4s}{d}<\alpha<s^*$.} Let $a>0$. We take $u \in H^s(\R^d)$ such that $\|u\|^2_{L^2} =a$. Let $u_\lambda$ be as in \eqref{scaling}. We see that $\|u_\lambda\|^2_{L^2} = \|u\|^2_{L^2} =a$ and
\[
E(u_\lambda) = \frac{\lambda^{2s}}{2} \|(-\Delta)^{\frac{s}{2}} u\|^2_{L^2} +\frac{1}{2} \int V(\lambda^{-1} x) |u(x)|^2 dx - \frac{\lambda^{\frac{d\alpha}{2}}}{\alpha+2} \|u\|^{\alpha+2}_{L^{\alpha+2}}.
\]
Since $\frac{d\alpha}{2} >2s$, we see that the right hand side tends to $-\infty$ as $\lambda \rightarrow \infty$. Note that the second term in the right hand side is bounded due to the fact that $V$ is bounded. This shows that $I(a)=-\infty$ and there is no minimizer for $I(a)$. The proof is complete.
\hfill $\Box$

\section{Blow-up behavior of minimizers} \label{section-blowup}
\setcounter{equation}{0}

In this section, we study the blow-up behavior of minimizers for $I(a)$ in the mass-critical case given in Theorem $\ref{theo-blowup}$ and Theorem $\ref{theo-blowup-refined}$.

\noindent {\bf Proof of Theorem $\ref{theo-blowup}$.} We first prove that $u_a$ blows up as $a \nearrow a^*$ in the sense of \eqref{blowup-meaning}. Assume that it is not true, then $(u_a)_{a \nearrow a^*}$ is a bounded sequence in $H^s(\R^d)$. Applying the concentration-compactness principle with the fact $\|u_a\|^2_{L^2} =a \nearrow a^*$ as $a \nearrow a^*$, there exists a subsequence still denoted by $(u_a)_{a \nearrow a^*}$ satisfying one of the three posibilities: vanishing, dichotomy and compactness. Using the same argument as in the proof of Theorem $\ref{theo-existence}$ together with the fact (see \eqref{est-Ia-periodic-potential-cri-1})
\[
\frac{I(a^*)}{a^*} = \frac{1}{2} \min V <\frac{1}{2} \inf \sigma((-\Delta)^s+V),
\]
we see that the vanishing and dichotomy cannot occur. Thus the compactness must occur, and there thus exist $x^0 \in [0,1]^d$ and $u \in H^s(\R^d)$ such that $u(\cdot-x^0)$ is a minimizer for $I(a^*)$ which is a contradiction.

Let $\varep_a$ be as in \eqref{blowup-meaning}. We define
\[
v_a(x):= \varep_a^{\frac{d}{2}} u_a(\varep_a x).
\]
It is easy to see that
\[
\|v_a\|^2_{L^2} =\|u_a\|^2_{L^2} =a, \quad \|(-\Delta)^{\frac{s}{2}} v_a\|^2_{L^2} = \varep_a^{2s} \|(-\Delta)^{\frac{s}{2}} u_a\|^2_{L^2} =1.
\]
By the fractional Gagliardo-Nirenberg inequality and  \eqref{est-Ia-periodic-potential-cri-1},
\begin{align*}
0 &\leq \|(-\Delta)^{\frac{s}{2}} v_a\|^2_{L^2}  - \frac{d}{d+2s} \|v_a\|^{\frac{4s}{d}+2}_{L^{\frac{4s}{d}+2}} \\
&=\varep_a^{2s} \left(\|(-\Delta)^{\frac{s}{2}} u_a\|^2_{L^2} - \frac{d}{d+2s} \|u_a\|^{\frac{4s}{d}+2}_{L^{\frac{4s}{d}+2}} \right) \\
&= \varep_a^{2s} \left( 2E(u_a)  -  \int V |u_a|^2 dx \right) \\
&\leq  \varep_a^{2s} (2I(a) - a\min V) \rightarrow 0
\end{align*}
as $a \nearrow a^*$ which implies
\begin{align*} 
\lim_{a \nearrow a^*} \|(-\Delta)^{\frac{s}{2}} v_a\|^2_{L^2}  - \frac{d}{d+2s} \|v_a\|^{\frac{4s}{d}+2}_{L^{\frac{4s}{d}+2}} =0.
\end{align*}
By Lemma $\ref{lem-compact-mini-fGN}$, there exist a subsequence still denoted by $(v_a)_{a \nearrow a^*}$ and a sequence $(y_a)_{a \nearrow a^*} \subset \R^d$ such that 
\[
v_a(\cdot +y_a) \rightarrow \gamma^{\frac{d}{2}} Q(\gamma \cdot) \text{ strongly in } H^s(\R^d)
\]
for some $\gamma>0$ as $a \nearrow a^*$. Since $\|(-\Delta)^{\frac{s}{2}} v_a\|_{L^2}=1$, it follows that $\gamma=1$. We next write 
\begin{align} \label{defi-xa-ya}
\varep_a y_a = x_a + z_a, \quad x_a \in [0,1]^d, z_a \in \Z^d.
\end{align}
Since $(x_a)_{a \nearrow a^*}$ is bounded in $\R^d$, up to a subsequence, $x_a \rightarrow x^0 \in [0,1]^d$ as $a \nearrow a^*$. We will show that $V(x^0) = \min V$. Indeed, by the fractional Gagliardo-Nirenberg inequality and \eqref{est-Ia-periodic-potential-cri-1}, 
\[
\min V \leq \frac{1}{a} \int V(x)|u_a(x)|^2 dx \leq \frac{2I(a)}{a} \rightarrow \min V
\]
as $a \nearrow a^*$. By the Fatou's lemma,
\begin{align}
\min V = \lim_{a \nearrow a^*} \frac{1}{a} \int V(x)|u_a(x)|^2 dx &= \lim_{a \nearrow a^*} \frac{1}{a} \int V(\varep_a x) |v_a(x)|^2 dx \nonumber \\
&= \lim_{a \nearrow a^*} \frac{1}{a} \int V(\varep_a x+\varep_a y_a) |v_a(x+y_a)|^2 dx \nonumber \\
&= \lim_{a \nearrow a^*} \frac{1}{a} \int V(\varep_a x+x_a) |v_a(x+y_a)|^2 dx \nonumber \\
&\geq \int \lim_{a \nearrow a^*} \frac{1}{a} V(\varep_a x+x_a) |v_a(x+y_a)|^2 dx \nonumber \\
&= V(x^0) \frac{1}{a^*} \int [Q(x)]^2 dx = V(x^0) \label{V-x0}
\end{align}
which implies that $V(x^0) = \min V$. In the last equality, we have used the fact that up to a subsequence, $v_a(\cdot+y_a)$ converges to $Q$ almost everywhere. 
\hfill $\Box$

The rest of this section is devoted to the proof of Theorem $\ref{theo-blowup-refined}$. Before giving the proof, we need the following lemmas.

\begin{lemma} \label{lem-energy-est-cri}
	Let $d\geq 1$, $0<s<1$ and $V \in C(\R^d)$ satisfy \emph{(V1)}, \emph{(V2)} and \emph{(V3)}. Then there exist positive constants $C_1<C_2$ independent of $a$ such that as $a \nearrow a^*$,
	\begin{align} \label{energy-est-cri}
	C_1 \beta_a^{\frac{p}{2s+p}} \leq \frac{I(a)}{a} \leq C_2 \beta_a^{\frac{p}{2s+p}},
	\end{align}
	where $\beta_a$ is given in \eqref{blowup-scaling}.
\end{lemma}

\begin{proof}
	Taking $x^0\equiv x_0$ with $x_0$ in (V3), we have from \eqref{est-Ia-periodic-potential-cri-2} that
	\[
	\frac{I(a)}{a} \leq \frac{d}{4s}\tau^{2s} \beta_a + o_\tau(1)
	\]
	as $\tau \rightarrow \infty$. The upper bound follows by taking $\tau = C\beta_a^{-\frac{1}{2s+p}}$ for some constant $C>0$. Note that the constant $C_2$ in \eqref{energy-est-cri} can be made as large as we want by enlarging the constant $C$.
	
	To see the lower bound, we need the following claim.
	\begin{claim} \label{claim}
	Let $d\geq 1$, $0<s<1$ and $V$ satisfy \emph{(V1)}, \emph{(V2)} and \emph{(V3)}. Let $\varep_a$ and $y_a$ be as in the proof of Theorem $\ref{theo-blowup}$ and write
	\[
	\varep_a  y_a = x_a + z_a, \quad x_a \in [0,1]^d, \quad z_a \in \Z^d.
	\]
	Then there exists a subsequence still denoted by $(x_a)_{a \nearrow a^*}$ such that $x_a \rightarrow x_0$, where $x_0$ is given in \emph{(V3)}. Moreover, there exists $C>0$ indenpendent of $a$ such that
	\begin{align} \label{blowup-periodic-cri-proof-1}
	\lim_{a \nearrow a^*} \varep_a^{-p} \int_{\R^d} V(x)|u_a(x)|^2 dx \geq C,
	\end{align}
	where $p$ is given in \emph{(V3)}.
	\end{claim}
	\noindent {\it Proof of Claim $\ref{claim}$.}
	Since $(x_a)_{a \nearrow a^*}$ is bounded in $\R^d$, up to a subsequence, $x_a \rightarrow x^0$ for some $x^0 \in [0,1]^d$ as $a \nearrow a^*$. We need to show that $x^0 \equiv x_0$, i.e. $V(x^0)=0$ which is done by the same argument as in \eqref{V-x0}. It remains to prove \eqref{blowup-periodic-cri-proof-1}. To this end, we use Fatou's lemma and the assumption (V3) to get
	\begin{align}  \label{fatou-app}
	\begin{aligned}
	\lim_{a \nearrow a^*} \varep_a^{-p} \int V(x) |u_a(x)|^2 dx & = \lim_{a \nearrow a^*} \varep_a^{-p} \int V(\varep_a x + x_a) |v_a(x +y_a)|^2 dx \\
	&\geq \int \lim_{a \nearrow a^*} \varep_a^{-p} V(\varep_a x + x_a) |v_a(x +y_a)|^2 dx \\
	& = \kappa \int \lim_{a \nearrow a^*} \left| x + \frac{x_a-x_0}{\varep_a}\right|^p [Q(x)]^2 dx.
	\end{aligned}
	\end{align}
	We will show that $\frac{x_a-x_0}{\varep_a}$ is uniformly bounded as $a \nearrow a^*$. Indeed, if $\left|\frac{x_a-x_0}{\varep_a}\right| \rightarrow \infty$ as $a \nearrow a^*$, then for any $N>0$ large, $\left|\frac{x_a-x_0}{\varep_a}\right| \geq N$ for $a$ close to $a^*$. It follows that
	\begin{align*}
	\lim_{a \nearrow a^*} \varep_a^{-p} \int V(x) |u_a(x)|^2 dx &\geq \kappa \int_{|x| \leq \frac{N}{2}} \lim_{a \nearrow a^*} \left|x+ \frac{x_a-x_0}{\varep_a}\right|^p [Q(x)]^2 dx \\
	&\geq \kappa \left(\frac{N}{2}\right)^p \int_{|x| \leq \frac{N}{2}} [Q(x)]^2 dx = C N^p,
	\end{align*}
	for some $C>0$ indenpendent of $N$. Using the fact $\|(-\Delta)^{\frac{s}{2}} v_a\|_{L^2}=1$ and $\|v_a\|^2_{L^2}=a$, we infer that as $a \nearrow a^*$,
	\begin{align}
	I(a) = E(u_a) &= \frac{1}{2} \|(-\Delta)^{\frac{s}{2}} u_a\|^2_{L^2} - \frac{d}{2(d+2s)} \|u_a\|^{\frac{4s}{d}+2}_{L^{\frac{4s}{d}+2}} + \frac{1}{2} \int V(x) |u_a(x)|^2 dx \nonumber \\
	& = \frac{1}{2} \varep_a^{-2s} \left( \|(-\Delta)^{\frac{s}{2}} v_a\|^2_{L^2} - \frac{d}{d+2s} \|v_a\|^{\frac{4s}{d}+2}_{L^{\frac{4s}{d}+2}} \right) + \frac{1}{2} \int V(x) |u_a(x)|^2 dx \nonumber \\
	& \geq \frac{1}{2} \varep_a^{-2s} \left(1- \left(\frac{a}{a^*} \right)^{\frac{2s}{d}} \right) + C \varep_a^p N^p \nonumber \\
	&= \frac{1}{2} \varep_a^{-2s} \beta_a + C \varep_a^p N^p \nonumber \\
	&\geq C(p, s) N^{\frac{2sp}{2s+p}} \beta_a^{\frac{p}{2s+p}}, \label{blowup-periodic-cri-proof-2}
	\end{align}
	where the last inequality follows by Young's inequality. This however contradicts to the upper bound in \eqref{energy-est-cri} by taking $N>0$ sufficiently large. It follows that $\left|\frac{x_a -x_0}{\varep_a}\right|$ is uniformly bounded as $a \nearrow a^*$, says $\lim_{a \nearrow a^*} \left|\frac{x_a -x_0}{\varep_a}\right| \leq C$. We see that
	\begin{align*}
	\lim_{a \nearrow a^*} \varep_a^{-p} \int V(x)|u_a(x)|^2 dx \geq \kappa \int_{|x|\geq 2C} \lim_{a \nearrow a^*} \left| x+\frac{x_a-x_0}{\varep_a}\right|^p [Q(x)]^2 dx \geq C
	\end{align*}
	for some $C>0$ independent of $a$. Note that the constant $C>0$ may change from line to line.	
	\hfill $\Box$

	We now prove the lower bound in \eqref{energy-est-cri}. By \eqref{blowup-periodic-cri-proof-1} and the same argument as \eqref{blowup-periodic-cri-proof-2}, we have that for $a\nearrow a^*$,
	\[
	I(a) \geq C\beta_a^{\frac{p}{2s+p}}.
	\]
	The lower bound follows by taking $C_1 = \frac{C}{a^*}$.
\end{proof}

\begin{lemma}
	Let $V \in C(\R^d)$ satisfy \emph{(V1)}, \emph{(V2)} and \emph{(V3)}. Let $u_a$ be a minimizer for $I(a)$ with $a_*<a<a^*$. Then there exist constants $C_3, C_4$ independent of $a$ such that as $a \nearrow a^*$,
	\begin{align} \label{nonlinear-est-cri}
	\frac{\|(-\Delta)^{\frac{s}{2}} u_a\|^2_{L^2}}{a} \leq C_3 \beta_a^{-\frac{2s}{2s+p}}, \quad C_4 \beta_a^{-\frac{2s}{2s+p}} \leq \frac{\|u_a\|^{\frac{4s}{d}+2}_{L^{\frac{4s}{d}+2}}}{a^{\frac{2s}{d}+1}}.
	\end{align}
\end{lemma}

\begin{proof}
	Using the fractional Gagliardo-Nirenberg inequality and the fact $V\geq 0$, we see that
	\begin{align*}
	I(a) = E(u_a) &\geq \frac{d}{2(d+2s)} \left( \left(\frac{a_*}{a}\right)^{\frac{2s}{d}} -1 \right) \|u_a\|^{\frac{4s}{d}+2}_{L^{\frac{4s}{d}+2}} \\
	&= \frac{d}{2(d+2s)} \frac{\beta_a}{1-\beta_a} \|u_a\|^{\frac{4s}{d}+2}_{L^{\frac{4s}{d}+2}} \\
	&\geq \frac{d}{2(d+2s)} \beta_a \|u_a\|^{\frac{4s}{d}+2}_{L^{\frac{4s}{d}+2}}.
	\end{align*}
	It follows from the upper bound in \eqref{energy-est-cri} and $V\geq 0$ that
	\begin{align*}
	\frac{\|(-\Delta)^{\frac{s}{2}} u_a\|^2_{L^2}}{a} &\leq 2\frac{I(a)}{a} + \frac{d}{d+2s} \frac{\|u_a\|^{\frac{4s}{d}+2}_{L^{\frac{4s}{d}+2}}}{a} \\
	&\leq 2\frac{I(a)}{a \beta_a} (\beta_a+1) \leq 4 C_2 \beta_a^{-\frac{2s}{2s+p}}
	\end{align*}
	which proves the first estimate in \eqref{nonlinear-est-cri}.
	
	Let $0<b<a$ be such that $b^{\frac{2s}{d}} = a^{\frac{2s}{d}} - \rho^{\frac{2s}{d}} \beta_a$ for some $\rho>0$ to be chosen later. Note that as $a \nearrow a^*$, $b$ stays close to $a$. Now set $\lambda = \sqrt{\frac{b}{a}}$. By the choice of $b$, we have 
	\[
	1 -\lambda^{\frac{4s}{d}} = \left(\frac{\rho}{a}\right)^{\frac{2s}{d}} \beta_a.
	\]
	This together with
	\[
	E(\lambda u_a) = \lambda^2 E(u_a) +\frac{\lambda^2\left(1-\lambda^{\frac{4s}{d}}\right)}{\frac{4s}{d}+2} \|u_a\|^{\frac{4s}{d}+2}_{L^{\frac{4s}{d}+2}}
	\]
	imply that
	\begin{align*}
	\frac{\left(\frac{\rho}{a}\right)^{\frac{2s}{d}} \beta_a}{\frac{4s}{d}+2} \lambda^2 \|u_a\|^{\frac{4s}{d}+2}_{L^{\frac{4s}{d}+2}} = E(\lambda u_a) - \lambda^2 E(u_a) \geq I(b) - \lambda^2 I(a).
	\end{align*}
	In particular, 
	\begin{align*}
	\frac{\rho^{\frac{2s}{d}} \beta_a}{\frac{4s}{d}+2}  \frac{\|u_a\|^{\frac{4s}{d}+2}_{L^{\frac{4s}{d}+2}}}{a^{\frac{2s}{d}+1}} &\geq  \frac{I(b)}{b} - \frac{I(a)}{a} \\
	&\geq C_1 \beta_b^{\frac{p}{2s+p}} - C_2 \beta_a^{\frac{p}{2s+p}} \\
	&\geq \beta_a^{\frac{p}{2s+p}} \left( C_1 \left(\frac{\beta_b}{\beta_a}\right)^{\frac{p}{2s+p}} - C_2 \right).
	\end{align*}
	By the choice of $b$, we see that
	\[
	\beta_b = 1-\left(\frac{b}{a^*}\right)^{\frac{2s}{d}} = \beta_a \left(1+\left(\frac{\rho}{a^*}\right)^{\frac{2s}{d}} \right).
	\]
	This implies that
	\[
	\frac{\rho^{\frac{2s}{d}}}{\frac{4s}{d}+2} \frac{\|u_a\|^{\frac{4s}{d}+2}_{L^{\frac{4s}{d}+2}}}{a^{\frac{2s}{d}+1}} \geq \beta_a^{-\frac{2s}{2s+p}} \left[ C_1 \left(1+\left(\frac{\rho}{a^*}\right)^{\frac{2s}{d}}\right)^{\frac{p}{2s+p}} - C_2 \right].
	\]
	By taking $\rho>0$ sufficiently large, the right hand side is bounded from below by $C \beta_a^{-\frac{2s}{2s+p}}$ for some constant $C$ independent of $a$. This proves the second estimate in \eqref{nonlinear-est-cri}.
\end{proof}

\noindent {\bf Proof of Theorem $\ref{theo-blowup-refined}$.}
Set 
\[
w_a(x):= \beta_a^{\frac{d}{2(2s+p)}} u_a\left( \beta_a^{\frac{1}{2s+p}} x \right).
\]
It follows that $\|w_a\|^2_{L^2} = \|u_a\|^2_{L^2} = a \nearrow a^*$ and by \eqref{nonlinear-est-cri},
\[
\|(-\Delta)^{\frac{s}{2}} w_a\|^2_{L^2} = \beta_a^{\frac{2s}{2s+p}} \|(-\Delta)^{\frac{s}{2}} u_a\|^2_{L^2} \leq C_3 a, \quad \|w_a\|^{\frac{4s}{d}+2}_{L^{\frac{4s}{d}+2}} = \beta_a ^{\frac{2s}{2s+p}} \|u_a\|^{\frac{4s}{d}+2}_{L^{\frac{4s}{d}+2}} \geq C_4 a^{\frac{2s}{d}+1}
\]
as $a \nearrow a^*$. Moreover, by \eqref{energy-est-cri},
\[
0\leq \frac{1}{a} \beta_a^{\frac{2s}{2s+p}} \left(\frac{1}{2} \|(-\Delta)^{\frac{s}{2}} u_a\|^2_{L^2} - \frac{d}{2(d+2s)} \|u_a\|^{\frac{4s}{d}+2}_{L^{\frac{4s}{d}+2}} \right) \leq \beta_a ^{\frac{2s}{2s+p}} \frac{I(a)}{a} \rightarrow 0
\]
as $a \nearrow a^*$. In particular,
\[
\lim_{a \nearrow a^*} \|(-\Delta)^{\frac{s}{2}} w_a\|^2_{L^2} - \frac{d}{d+2s} \|w_a\|^{\frac{4s}{d}+2}_{L^{\frac{4s}{d}+2}} = 0.
\]
Applying Lemma $\ref{lem-compact-mini-fGN}$ to $(w_a)_{a \nearrow a^*}$, there exist a subsequence still denoted by $(w_a)_{a \nearrow a^*}$ and a sequence $(y_a)_{a \nearrow a^*} \subset \R^d$ such that 
\begin{align} \label{convergence-wa}
w_a(\cdot+y_a) \rightarrow \lambda_0^{\frac{d}{2}} Q(\lambda_0 \cdot) \text{ strongly in } H^s(\R^d)
\end{align}
for some $\lambda_0>0$ as $a \nearrow a^*$. We next write $\beta^{\frac{1}{2s+p}} y_a = x_a + z_a$ with $x_a \in [0,1]^d$ and $z_a \in \Z^d$. By the same argument as in Claim $\ref{claim}$, $x_a \rightarrow x_0$, where $x_0$ is as in (V3). Moreover, $\beta_a^{-\frac{1}{2s+p}}(x_a-x_0)$ is uniformly bounded as $a \nearrow a^*$. Passing to a subsequence if necessary, we assume that $\beta_a^{-\frac{1}{2s+p}}(x_a-x_0) \rightarrow x^0 \in \R^d$. By the same argument as in \eqref{fatou-app}, we have that
\begin{align}
\lim_{a \nearrow a^*} \beta_a^{-\frac{p}{2s+p}} \int V(x) |u_a(x)|^2 dx &= \lim_{a \nearrow a^*} \beta_a^{-\frac{p}{2s+p}} \int V\left(\beta_a^{\frac{1}{2s+p}} x+ x_a\right) |w_a(x+y_a)|^2 dx \nonumber \\
&\geq \int \lim_{a\nearrow a^*} \beta_a^{-\frac{p}{2s+p}} V\left(\beta_a^{\frac{1}{2s+p}} x+ x_a\right) |w_a(x+y_a)|^2 dx \nonumber \\
&= \kappa \int \lim_{a \nearrow a^*} \left|x + \beta_a^{-\frac{1}{2s+p}}(x_a-x_0) \right|^p \lambda_0^d [Q(\lambda_0 x)]^2 dx \nonumber \\
&= \kappa \int |x+x^0|^p \lambda_0^d [Q(\lambda_0 x)]^2 dx \nonumber \\
&= \kappa \lambda_0^{-p} \int |x+ x^0 \lambda_0|^p [Q(x)]^2 dx \nonumber \\
&\geq \kappa \lambda_0^{-p} \int |x|^p [Q(x)]^2 dx, \label{est-x^0}
\end{align}
where we use the fact $Q$ is radially symmetric and decreasing to get the last inequality. Note that the equality holds if and only if $x^0 \equiv 0$. We infer from the above estimate, \eqref{convergence-wa} and the fractional Gagliardo-Nirenberg inequality that
\begin{align*}
\beta_a^{-\frac{p}{2s+p}} \frac{E(u_a)}{a} &= \frac{1}{2a} \beta_a^{-\frac{p}{2s+p}} \left(\|(-\Delta)^{\frac{s}{2}} u_a\|^2_{L^2} - \frac{d}{d+2s} \|u_a\|^{\frac{4s}{d}+2}_{L^{\frac{4s}{d}+2}} \right) +  \frac{1}{2a} \beta_a^{-\frac{p}{2s+p}} \int V(x) |u_a(x)|^2 dx \\
&= \frac{1}{2a} \beta_a^{-1}\left(\|(-\Delta)^{\frac{s}{2}} w_a\|^2_{L^2}-\frac{d}{d+2s} \|w_a\|^{\frac{4s}{d}+2}_{L^{\frac{4s}{d}+2}} \right) + \frac{1}{2a} \beta_a^{-\frac{p}{2s+p}} \int V(x) |u_a(x)|^2 dx \\
&\geq \frac{1}{2a} \beta_a^{-1} \left(1-\left(\frac{a}{a^*}\right)^{\frac{2s}{d}} \right) \|(-\Delta)^{\frac{s}{2}} w_a\|^2_{L^2} + \frac{1}{2a}\beta_a^{-\frac{p}{2s+p}} \int V(x) |u_a(x)|^2 dx \\
&=\frac{1}{2a^*} \left(\lambda_0^{2s} \|(-\Delta)^{\frac{s}{2}} Q\|^2_{L^2} + \kappa \lambda_0^{-p} \int |x|^p [Q(x)]^2 dx \right) + o_{a \nearrow a^*}(1) \\
&= \frac{\lambda_0^{2s} d}{4s} + \frac{1}{2} \kappa \lambda_0^{-p} \int |x|^p [Q_0(x)]^2 dx + o_{a \nearrow a^*}(1),
\end{align*}
where we have used that $Q_0 = \frac{Q}{\|Q\|_{L^2}}$, $\|(-\Delta)^{\frac{s}{2}} Q_0\|^2_{L^2}=\frac{d}{2s}$. It follows that
\begin{align} \label{liminf-Ia}
\liminf_{a \nearrow a^*} \beta_a^{-\frac{p}{2s+p}} \frac{I(a)}{a} \geq \frac{\lambda_0^{2s} d}{4s} + \frac{1}{2} \kappa \lambda_0^{-p} \int |x|^p [Q_0(x)]^2 dx.
\end{align}
We next show that the limit in \eqref{liminf-Ia} exists, and is equal to the right hand side. To see this, we take
\[
u_\tau(x) = \sqrt{a} \tau^{\frac{d}{2}} Q_0(\tau(x-x_0)),
\]
where $x_0$ is given in (V3). We see that
\begin{align*}
\frac{E(u_\tau)}{a} &= \frac{1}{2} \left[ \tau^{2s} \left(\|(-\Delta)^{\frac{s}{2}} Q_0\|^2_{L^2} - \frac{d}{d+2s} a^{\frac{2s}{d}} \|Q_0\|^{\frac{4s}{d}+2}_{L^{\frac{4s}{d}+2}} \right) + \int V(\tau^{-1}x+x_0) [Q_0(x)]^2 dx \right] \\
&=\frac{1}{2} \left[ \tau^{2s} \left(1-\left(\frac{a}{a^*}\right)^{\frac{2s}{d}} \right) \|(-\Delta)^{\frac{s}{2}} Q_0\|^2_{L^2} +  \int V(\tau^{-1}x+x_0) [Q_0(x)]^2 dx \right] \\
&=\tau^{2s} \beta_a \frac{d}{4s} + \frac{\kappa}{2} \tau^{-p} \int |x|^p [Q_0(x)]^2 dx + \frac{1}{2} \int (V(\tau^{-1}x+x_0)- \kappa \tau^{-p} |x|^p) [Q_0(x)]^2 dx.
\end{align*}
Taking $\tau= \lambda \beta_a^{-\frac{1}{2s+p}}$ for $\lambda>0$, we get that
\begin{align*}
\beta_a^{-\frac{p}{2s+p}} \frac{E(u_\tau)}{a} = \frac{\lambda^{2s} d}{4s} &+ \frac{\kappa}{2} \lambda^{-p} \int |x|^p [Q_0(x)]^2 dx \\
&+ \frac{1}{2} \int \left[ \beta_a^{-\frac{p}{2s+p}} V\left(\beta_a^{\frac{1}{2s+p}} \lambda^{-1} x+ x_0 \right) - \kappa \lambda^{-p} |x|^p\right] [Q_0(x)]^2 dx.
\end{align*}
Using (V3) and the fact $Q_0 = O(|x|^{-d-2s})$ as $|x| \rightarrow \infty$, it is easy to see that the integral term tends to zero as $a \nearrow a^*$ due to the dominated convergence theorem. It yields that for any $\lambda>0$,
\[
\beta_a^{-\frac{p}{2s+p}}\frac{I(a)}{a} \leq \frac{\lambda^{2s} d}{4s} + \frac{\kappa}{2} \lambda^{-p} \int |x|^p [Q_0(x)]^2 dx + o_{a \nearrow a^*}(1).
\]
This implies that
\[
\limsup_{a \nearrow a^*} \beta_a^{-\frac{p}{2s+p}}\frac{I(a)}{a} \leq \inf_{\lambda>0} \left(\frac{\lambda^{2s} d}{4s} + \frac{\kappa}{2} \lambda^{-p} \int |x|^p [Q_0(x)]^2 dx \right).
\]
This together with \eqref{liminf-Ia} show that
\[
\lim_{a \nearrow a^*} \beta_a^{-\frac{p}{2s+p}}\frac{I(a)}{a} =\frac{\lambda_0^{2s} d}{4s} + \frac{\kappa}{2} \lambda_0^{-p} \int |x|^p [Q_0(x)]^2 dx = \inf_{\lambda>0} \left(\frac{\lambda^{2s} d}{4s} + \frac{\kappa}{2} \lambda^{-p} \int |x|^p [Q_0(x)]^2 dx \right)
\]
and also $x^0 \equiv 0$ because of the equality in \eqref{est-x^0}. Therefore, we obtain
\[
\lambda_0 = \left( \frac{\kappa p}{d} \int |x|^p [Q_0(x)]^2 dx \right)^{\frac{1}{2s+p}}.
\]
The proof is complete.
\hfill $\Box$	

\section*{Acknowledgments}
This work was supported in part by the Labex CEMPI (ANR-11-LABX-0007-01). The author would like to express his deep gratitude to his wife - Uyen Cong for her encouragement and support. He also would like to thank the reviewer for his/her helpful comments and suggestions.


\end{document}